\documentclass[a4paper,11pt]{article}
\usepackage[pagewise]{lineno}
\usepackage{amsmath}
\usepackage{amssymb}
\usepackage{mathrsfs}
\usepackage{amsfonts}
\usepackage{amsthm}
\usepackage{pstricks, pst-node, pst-text, pst-3d,psfrag}
\usepackage{graphicx}
\usepackage{indentfirst}
\usepackage{enumerate}
\usepackage{subfigure}
\usepackage{cite}
\usepackage[colorlinks=true]{hyperref}
\hypersetup{urlcolor=blue, citecolor=red}
\usepackage{soul}
\usepackage{footmisc}
\usepackage{authblk}

\theoremstyle{plain}
\newtheorem{thm}{Theorem}[section]
\newtheorem{prop}[thm]{Proposition}
\newtheorem{thmm}{Theorem}[subsection]
\newtheorem{cor}[thm]{Corollary}
\newtheorem{coro}{Corollary}[subsection]
\newtheorem{lem}[thm]{Lemma}
\theoremstyle{definition}
\newtheorem{defn}[thm]{Definition}

\newtheorem{rmk}[thm]{Remark}

\numberwithin{equation}{section}

\makeatletter % `@' now normal "letter"
\@addtoreset{equation}{section}
\makeatother  % `@' is restored as "non-letter"

\begin{document}
	
	\title{Birkhoff center and recurrent behavior of differentially positive systems on a homogeneous space}

	\setlength{\baselineskip}{16pt}
	
	\author[1]{Lin Niu\footnote{Supported by the National Natural Science Foundation of China No.12201034.}}
	\author[2]{Yi Wang\footnote{Supported by the National Key R\&D Program of China (2024YFA1013603, 2024YFA1013600), NSFC No.12331006, and Strategic Priority Research Program of Chinese Academy of Sciences (XDB0900100).}}
	\author[3]{Yufeng Zhang\footnote{Supported by the National Natural Science Foundation of China (No.12401240).}} 
	
	% 单位列表，用 \affil 声明，[1]、[2] 对应作者的角标
	\affil[1]{\footnotesize{School of Mathematics and Physics, University of Science and Technology Beijing, Beijing, 100083, P. R. China}}
	\affil[2]{\footnotesize{School of Mathematical Sciences, University of Science and Technology of China, Hefei, Anhui, 230026, P. R. China}}
	\affil[3]{\footnotesize{Department of Mathematical and Statistical Sciences, University of Alberta, Edmonton, AB T6G 2G1, Canada}}

\date{}

    \maketitle
	
	%---------------------SECTION DIVIDE LINE---------------------------
	\begin{abstract}
		We study the structure of the Birkhoff center and the recurrent behavior of differentially positive systems whose linearizations along trajectories preserve a homogeneous cone field on a homogeneous space. Such cone fields arise naturally from general relativity and Lie theory. We establish an order-structural dichotomy for the Birkhoff center: every connected component of the Birkhoff center, as well as the support of any invariant measure, is either strongly ordered or unordered. This yields a comprehensive characterization of recurrent dynamics for differentially positive systems on homogeneous spaces, interpreted through the lens of the underlying order relation.
		
		\vskip 3mm
		
	\end{abstract}
	
	
\section{Introduction}
	
We investigate problems formulated on spaces endowed with special geometries, such as homogeneous spaces. These are manifolds that admit a transitive Lie group action. A fundamental property of such spaces is that additional structures defined on the manifold are invariant, or homogeneous, with respect to the underlying symmetry. A prototypical instance of such an additional geometric structure is that of a homogeneous cone fields on a given homogeneous space (see, e.g., \cite{HilgertNeeb93,HilgertHofmannLawson,Neeb91}). Roughly speaking, there are at least two motivations for studying homogeneous cone fields (see \cite{Lawson89}). First, in the causality theory of spacetimes, naturally arising symmetry conditions may lead one precisely to this situation. Second, historically, it was Lie's idea that finding the groups under which a differential equation is invariant could be useful for analyzing that differential equation. If we view cone fields on a manifold as a generalization of vector fields and seek curves consistent with them, then a similar motivation for considering invariance arises in this setting.

In the present paper, we consider a class of nonlinear dynamical systems for which the linearizations along trajectories preserve a homogeneous cone field on a homogeneous space. Such systems are called differentially positive systems (see, e.g., \cite{ForniandSepulchre16,NW25}). The concept of differential positivity was first introduced by Forni and Sepulchre \cite{ForniandSepulchre14,ForniandSepulchre16} in their study of the dynamics for individual orbits of nonlinear systems such as the pendulum model. Mostajeran and Sepulchre \cite{MostajeranSepulchre18homogeneous,MostajeranSepulchre18SIAM} further studied the differential positivity on homogeneous spaces. In \cite[Section 4]{MostajeranSepulchre18SIAM}, the differential positivity is applied to consensus protocols on the $N$-torus to analyze synchronization phenomena in smart grids and engineered complex oscillator networks. Recently, the present authors \cite{NW25,NWZ26} characterized the generic behavior of differentially positive systems, mainly from the perspective of $\omega$-limit sets for individual orbits.

In the theory of dynamical systems, the presence of non-periodic recurrence is widely recognized as the first indication of complicated asymptotic behavior. To capture the essence of non-periodic recurrence, the Birkhoff center, defined as the closure of the set of all recurrent points, serves as a pivotal hub of recurrence behavior (cf.\cite{K95,M12}). It constitutes not only the ultimate destination of recurrent behaviors but also the fundamental skeleton of recurrence structure. Research on the Birkhoff center is primarily concerned with the statistical behavior of orbits and their recurrence properties under invariant measures. Notably, the well-known Poincaré recurrence theorem establishes that invariant measures are concentrated on the Birkhoff center (see, e.g., \cite{J20,M12}).

The purpose of the present paper is to provide a characterization of the structure of the Birkhoff center and give a comprehensive description of the recurrence behavior for differentially positive systems on a homogeneous space.
For this purpose, we begin by recalling some insights for the cone fields on a smooth manifold. In a general way, a cone field, assigning to each point of a smooth manifold $M$ a cone in the tangent space, naturally provides a local or differential way of thinking about order relations on $M$. Meanwhile, the order relation itself can be viewed as a corresponding global concept. It is the so-called conal curves that serve as a bridge connecting the cone field and the order relation on $M$. Here, a conal curve is a continuous and piecewise continuously differentiable curve whose tangent vector lies in the cone at every point along the curve wherever it is defined. As a matter of fact, a cone field on $M$ induces a ``conal order relation" as follows: two points $x,y\in M$ are {\it conal ordered}, denoted by $x\leq_M y$, if there exists a conal curve on $M$ beginning at $x$ and ending at $y$.
This order relation on a homogeneous space provides a general framework for investigating the differentially positive systems.

The aforementioned idea of the order relation originally stems from the causality theory of spacetimes studied by Hawking \cite{Hawkingandellis73}, Penrose \cite{Penrose72} and others \cite{Beemandehrlich81,Wald84}, where a time-orientable spacetime naturally determines, at each point, a Lorentzian cone in the tangent space, i.e., a closed, convex, pointed cone representing all future-directed non-spacelike directions. Such cone field defines the class of non-spacelike (or causal) curves, which are precisely those curves whose tangent vectors lie in the cone at each point along the curve.
In this setting, each point of spacetime corresponds to an event, and a signal can be sent from $p$ to $q$ if there exists a future-directed causal curve from $p$ to $q$. That is, causally related points in spacetime are connected by a non-spacelike curve determined by the Lorentzian cone field.
This causality theory provides significant motivation and guide for the present work.

From this point of view, the conal order ``$\leq_M$" induced by a cone field on a manifold $M$ is highly related to the manifold's topological and geometric structure and the cone field. Due to this, ``$\leq_M$" may be not a global partial order. In fact, there are examples of $M$ that contain the closed conal curves (e.g., the closed timelike curves in time-orientable space-times \cite[Chapter 5]{Hawkingandellis73}), which reveals that the antisymmetry of the conal order relation ``$\leq_M$" fails.
Whether the conal order associated with a cone field can be extended to a partial order on the global manifold is a central problem in \cite{Lawson89}. The equivalence between globality of the conal order in a homogeneous space and globality of the Lie wedge in the Lie group has been shown in \cite[Section 5]{Lawson89} (see also \cite[Theorem 1.6]{Neeb91} and \cite[Section 4.3]{HilgertNeeb93}). It is known that the differential positivity is indeed the monotonicity when the cone field induces a global partial order, and hence, differentially positive systems can be regarded as a natural generalization of the so-called classical monotone systems from flat spaces to nonlinear manifolds (see \cite[Appendix A.3]{NW25}). By the way, from the perspective of Lie theory, the recent works in \cite{Neebolafsson21,Neebolafsson21b,Neebolafsson23,Morinellineeb21,Neeborstedolafsson21}
form an ongoing project to study the causal structures (represented by homogeneous cone fields) on homogeneous spaces.

Precisely due to the absence of antisymmetry in the order relation (i.e., we do not exclude the occurrence of closed conal curves in this article), new dynamical phenomena then arise: If $x$ is a recurrent point and the orbit of $x$ is so-called pseudo-ordered (i.e., there is some $T>0$ such that $x\neq \varphi_T(x)$ and either $x\leq_M \varphi_T(x)$ or $\varphi_T(x)\leq_M x$ holds), then the omega limit set $\omega(x)$ of $x$ is ordered, furthermore, for any $y_1, y_2 \in \omega(x)$, both $y_1\leq_M y_2$ and $y_2\leq_M y_1$ hold, see Proposition \ref{pseudoorder-orbit}. 
(As a comparison, for a partial order in the classical monotone theory, every pseudo-ordered orbit converges to an equilibrium due to the monotone convergence criterion, see, e.g., \cite[Theorem 1.2.1]{H95}.)
This new phenomenon underlies our discussion of the properties of the limit sets, recurrence sets, and the Birkhoff center.

Meanwhile, the relation ``$\leq_M$" is not necessarily closed, i.e., the set $\{(x,y)\in M\times M: x\leq_M y \}$ is not a closed subset in $M\times M$ (see Lawson \cite[p.299]{Lawson89} or Neeb \cite[p.470]{Neeb91}). For instance, such a set is not necessarily closed in Minkowski space (see Hawking and Ellis\cite[p.183]{Hawkingandellis73} or Penrose \cite[p.12]{Penrose72}).

To the best of our knowledge, there are several ways to deal with this feature of the order relation. One is to directly discuss the closure of the order relation ``$\leq_M$" (see, e.g., \cite[Section 7]{Lawson89} and \cite{Neeb91}). Another is to impose causality conditions in the causality theory of spacetimes that force the relation ``$\leq_M$" to be closed, for example, causal simplicity or global hyperbolicity (see \cite{Hawkingandellis73, Beemandehrlich81}).
However, both of the above approaches have shortcomings: In general it is not true that the closure of an order is again a transitive relation (see \cite[p.470]{Neeb91}), and 
a causal simple or globally hyperbolic spacetime is strongly causal which excludes the occurrence of closed timelike curves (see \cite[p.73]{Beemandehrlich81}).

In the present work, we focus on the so-called strongly differentially positive system (see Definition \ref{differential positive}), whose flow $\varphi_t$ naturally preserves the conal order (i.e., if $x\leq_M y$ with $x\neq y$, then $\varphi_t(x)\ll_M \varphi_t(y)$ for all $t>0$, see Proposition \ref{monotone property}). Here, we write $x\ll_M y$ (see Definition \ref{def-order}) if there exists a so-called \textit{strictly} conal curve (see Definition \ref{conal-curve}), whose tangent vector lies in the interior of the cone at every point along the curve, on $M$ beginning at $x$ and ending at $y$. Therefore, when studying the asymptotic behavior of a strongly differentially positive system from the perspective of the order relation, we focus more on the properties of the strong order relation ``$\ll_M$" than the order relation ``$\leq_M$". For this reason, we first introduce the following assumption:
\begin{enumerate}
	\item[{\bf (H2)}] The conal order ``$\leq_M$" is quasi-closed. 
\end{enumerate}
\noindent Here, the quasi-closedness of ``$\leq_M$" in (H2) means that {\it $x\leq_M y$ whenever $x_n\to x$ and $y_n\to y$ as $n\to \infty$ and $x_n\ll_M y_n$ for all integer $n\ge 0$} (see Definition \ref{quasi-closed}).
Roughly speaking, Hypothesis (H2) is a suitable substitute for the closedness of the order relation ``$\leq_M$". %Refer to \cite{NW25} for more details.

The strong order relation ``$\ll_M$", more precisely, the strictly conal curve on $M$ is inspired by the future directed timelike curves in the causality theory of spacetimes. Furthermore, the events in a spacetime that can be reached by a material particle starting at $p$ comprise the ``chronological future" of $p$, denoted by $I^+(p)$. In a general way, for $p\in M$,
\begin{itemize}
	\item chronological future set $I^+(p)=\{q\in M \colon p\ll_M q\}$,
	\item chronological past set $I^-(p)=\{q\in M \colon q\ll_M p\}$.
\end{itemize}
Past and future sets have played an important role in singularity theory in general relativity as treated in Penrose \cite{Penrose72} or Hawking and Ellis \cite{Hawkingandellis73} via the allied concept of chronological boundaries (i.e., the sets $\partial I^{\pm}(p)$).

In this paper, for every point $p$ in $M$, we assume that
\begin{enumerate}
	\item[{\bf (H3)}] The chronological sets $I^+$ and $I^-$ are continuous.
\end{enumerate}
The definition of continuity of the chronological sets (see Definition \ref{continuity-for-chronological sets}) originates from Hawking and Sachs \cite[Section 1.6]{Hawkingandsachs74}, where they provided a crucial link between the continuity of the chronological sets and that of the time functions.
In the present work, hypothesis (H3) enables us to discuss the properties of the boundary of the chronological sets and to obtain the so-called \textit{intersection principle for the chronological boundary} (see Lemma \ref{endpoint}).

Based on this intersection principle and the assumptions above, we present our main results. We establish \textit{an order-structural dichotomy} of the Birkhoff center for differentially positive systems on a homogeneous space. Here, a set $B$ is said to be \textit{strongly ordered} if any two points in $B$ are strongly ordered (w.r.t., ``$\ll_M$") and \textit{unordered} if any two points in $B$ are unordered (w.r.t., ``$\leq_M$"), see Definition \ref{def-order}.

\renewcommand{\thethmm}{\Alph{thmm}}

\begin{thmm}\label{Thm}
	Let $B$ be a connected component of Birkhoff center. Then $B$ is either strongly ordered or unordered.
\end{thmm}

This theorem, in a detailed version (see Theorem \ref{Thm_2.1}), will be proved in Section \ref{Limit dich}. It concludes that recurrence in differentially positive systems on a homogeneous space is restricted to strongly ordered or unordered components, i.e., in a connected component of Birkhoff center, either any two points are strongly ordered (w.r.t., ``$\ll_M$") or any two points are unordered (w.r.t., ``$\leq_M$").

As we mentioned that invariant measures are concentrated on the Birkhoff center,
a straightforward consequence of Theorem \ref{Thm} is the description of the order structure of the supports of the invariant measure on homogeneous space $M$. Here, the support of measure is the smallest closed set whose complement has measure zero.

\renewcommand{\thecoro}{\Alph{coro}}
\setcounter{coro}{1}

\begin{coro}
	Let $B$ be a connected component of the support of an invariant measure on $M$. Then $B$ is either strongly ordered or unordered.
\end{coro}

This corollary, in a more detailed version (see Corollary \ref{Thm_2.2}), establishes  that the support of an invariant measure, which captures the recurrent dynamics, adheres to an order configuration: either any two points are strongly ordered (w.r.t., ``$\ll_M$") or any two points are unordered (w.r.t., ``$\leq_M$").

It is evident that differentially positive systems provide a natural generalization of classical monotone systems, extending them from flat spaces to nonlinear manifolds. In the flat setting, Theorem \ref{Thm_2.1} and Corollary \ref{Thm_2.2} reduce to a celebrated theorem of Hirsch (see \cite[Theorem 1.6]{H99} and \cite{Mi00}), which asserts that every connected component $B$ of the Birkhoff center (or more generally, any attractor-free set $B$) is either unordered or a stationary p-arc. This theorem has been widely used to study the asymptotic behavior of stochastic approximation and the stochastic stability of differential equations (see \cite{B00,BF21,BH99,JWS26} and the references therein).

It should be noted that the Birkhoff center and the supports of invariant measures characterize the recurrence and statistical behavior, which cannot be obtained solely through the analysis of the limit sets of individual orbits. To establish the order-structural dichotomy of the Birkhoff center (the supports of invariant measures), a key innovation of this paper is to analyze the dynamics on the chronological boundary. Indeed, our approach relies heavily on this analysis, leading to a key lemma ``intersection principle" (Lemma \ref{endpoint}) about it. Meanwhile, we have to emphasize that the proof of this intersection principle is highly nontrivial and relies heavily on the homogeneous structure of the space, including homogeneous cone fields and homogeneous metrics (see Section \ref{s2}). To get the intersection principle, we study the order structure of $\omega$-limit sets (see Subsection \ref{s3.1}) and present a soft version of intersection principle (Lemma \ref{sep_n}). To reinforce the soft version to the key lemma, we establish a limit-set dichotomy for recurrent points (Lemma \ref{limit_set_dichotomy}), which helps us accomplish our approach.
	
\vskip 4mm

The paper is organized as follows. In Section \ref{s2}, we introduce the notations for homogeneous cone fields and related structures on homogeneous spaces, along with some preliminary results. Section \ref{DP} is devoted to introducing differentially positive systems and presenting the main results. To prove our main results, we study the dynamics on the chronological boundary of any recurrent point and present a fundamental lemma, the intersection principle, in Section \ref{s3}. Finally, the proof of Theorem \ref{Thm} (i.e., Theorem \ref{Thm_2.1}) is provided in Section \ref{Limit dich}.

\noindent\section{Homogeneous cone fields and chronological sets}\label{s2}

Let $G$ be a connected Lie group and $M$ be a smooth manifold. A \textit{left action} of Lie group $G$ on manifold $M$ is a smooth map $\theta \colon G \times M \to M$ satisfying $\theta(e,x)=x$ and $\theta(g_1g_2,x)=\theta(g_1,\theta(g_2,x))$ for all $g_1,g_2\in G$, $x\in M$, where $e$ is the identity element in $G$. We write $g\cdot x$ or $gx$ for $\theta(g,x)$. The action is said to be \textit{transitive} if for every pair of points $x,y\in M$, there exists $g\in G$ such that $g\cdot x=y$.

A smooth manifold endowed with a transitive smooth action by a Lie group $G$ is called a \textit{homogeneous G-space} (or a \textit{homogeneous space}).

Let $G$ be a Lie group and $M$ be a homogeneous $G$-space. If $p$ is a point of $M$, then the isotropy group $G_p=\{g\in G \colon g\cdot p=p\}$ is a closed subgroup of $G$, and the map $F \colon G/G_p \to M$ defined by $F(gG_p)=g\cdot p$ is an equivariant diffeomorphism (see \cite[Theorem 21.18]{LeeGTM218}), here, $G/G_p=\{gG_p \colon g\in G\}$ is the left coset space. Because of this equivariant diffeomorphism, we can define a homogeneous space to be a coset space of the form $G/H$, where $G$ is a Lie group and $H$ is a closed subgroup of $G$.

Let $G$ be a Lie group and fix $a\in G$. Define the \textit{left translation} map $L_a \colon G \to G$ by $L_a(g)=ag$. The left translation map is a diffeomorphism since it is smooth with smooth inverse. The inverse of $L_a$ is clearly the map $L_{a^{-1}}$, where $a^{-1}$ is the inverse of $a$. 
%The diffeomorphism $L_g$ induces a vector space isomorphism $dL_g\vert_e \colon \mathfrak{g}=T_eG \to T_gG$, where $\mathfrak{g}$ is the Lie algebra of $G$ and $e$ is the identity element in $G$.

Let $M=G/H$ be a homogeneous space, and let the natural projection $\pi \colon G\to G/H$ be defined by $\pi(g)=gH$. This projection is always a submersion. For each $a\in G$, define the \textit{left translation} $\lambda_a \colon G/H \to G/H$ by $\lambda_a(gH)=agH$. Then the left translations $\lambda_g$ are related to the left translations $L_g$ on the Lie group $G$ by $\pi \circ L_g = \lambda_g \circ \pi$ for each $g\in G$.

\vskip 4mm

A Riemannian metric on a smooth manifold $N$ is a correspondence which associates to each point $p\in N$ an inner product $(\cdot,\cdot)_p$ (that is a symmetric bilinear, positive definite form) on the tangent space $T_{p}N$. Sometimes, we write the metric as $(\cdot,\cdot)^{N}_p$ or $(\cdot,\cdot)^{N}$ to emphasize the manifold $N$.

A Riemannian metric $(\cdot,\cdot)^{G}$ on a Lie group $G$ is called \textit {left-invariant} if for each $a\in G$ the diffeomorphism $L_a$ is an isometry, that is,
$$(u,v)^{G}_x=(dL_a\vert_x(u),dL_a\vert_x(v))^{G}_{L_a(x)}$$ for all $x\in G$ and $u,v\in T_{x}G$.

Let $M=G/H$ be a homogeneous space with base-point $o=eH$. A Riemannian metric $(\cdot,\cdot)^{M}$ on $M$ is called \textit{G-invariant} if for each $a\in G$ the diffeomorphism $\lambda_a$ that sends $p\in M$ to $ap$ is an isometry, that is, $$(X,Y)^{M}_{o}=(d\lambda_a(X),d\lambda_a(Y))^{M}_{ao}$$ for all $X,Y\in T_oM$.

Throughout this paper, we assume that homogeneous space $M=G/H$ is endowed with a $G$-invariant Riemannian metric $(\cdot,\cdot)^{M}$ (which induces a Riemannian distance $d$ on $M$) and Lie group $G$ is endowed with a left-invariant Riemannian metric $(\cdot,\cdot)^{G}$ (which induces a Riemannian distance $d_G$ on $G$). Clearly, $(M,d)$ and $(G,d_G)$ are two complete metric spaces. For any $x,y\in M$, we denote by $g_{x}^{y}$ the group element that maps $x$ to $y$, i.e., $g_{x}^{y}x=y$. Then we have the following proposition:

\begin{prop}\label{distance-limit}
	Fix $x_0\in M$, for $x\in M$, there exists $g_{x_0}^{x}\in G$ such that
	\begin{enumerate}[{\rm(i)}]
		\item $d(x,x_0)\to 0$ $\Leftrightarrow$ $d_G(g_{x_0}^{x},e) \to 0$;
		\item $d(x,x_0)\to 0$ $\Leftrightarrow$ $d_G(g_{x}^{x_0},e) \to 0$.
	\end{enumerate}
   Moreover, fix $x_0,y_0\in M$, for $x\in M$, there exist $g_{x}^{y_0},g_{x_0}^{y_0}\in G$ such that
   \begin{itemize}
   	    \item[{\rm(iii)}] $d(x,x_0)\to 0$ $\Leftrightarrow$ $d_G(g_{x}^{y_0},g_{x_0}^{y_0}) \to 0$.
   \end{itemize}
\end{prop}
\begin{proof}
	For (i), ``$\Leftarrow$" is clear since $\theta(e,x_0)=x_0$ and $\theta$ is continuous at $e$. ``$\Rightarrow$" is obtained by the fact that for Lie group $G$ and isotropy group $G_{x_0}=\{g\in G \colon g\cdot x_0=x_0\}$, the quotient map $\pi \colon G\to G/G_{x_0}$ defined by $\pi(g)=gG_{x_0}$ is a smooth submersion (see, e.g., \cite[Chapter 21]{LeeGTM218}). For (ii), the left-invariant Riemannian metric on $G$ implies that $d_G(g_{x}^{x_0},e)=d_G(g_{x_0}^{x}g_{x}^{x_0},g_{x_0}^{x})=d_G(e,g_{x_0}^{x})$. So, the conclusion (ii) is obtained by (i). For (iii), the left-invariant Riemannian metric on $G$ also implies that $d_G(g_{x}^{y_0},g_{x_0}^{y_0})=d_G(g^{x}_{y_0}g_{x}^{y_0},g^{x}_{y_0}g_{x_0}^{y_0})=d_G(e,g_{x_0}^{x})$. Thus, we have (iii) by (i).
\end{proof}

\subsection{Homogeneous cone fields}

Let $E$ be a $r$-dimensional real linear space. A nonempty closed subset $C$ of the linear space $E$ is called a \textit{closed convex cone} if $C+C\subset C$, $\alpha C\subset C$ for all $\alpha \geq 0$, and $C \cap (-C) = \{0\}$. A convex cone $C$ is \textit{solid} if its interior $\text{Int} C \neq \emptyset$. A convex cone $C'$ in $E$ is said to \textit{surround} a cone $C$ if $C\setminus \{ 0\}\subset \text{Int}C'$.

A \textit{cone field} on a homogeneous space $M$ is a map $x \mapsto C_{M}(x)$, such that $C_{M}(x)$ is a closed convex cone in $T_{x}M$ for each $x\in M$.
A manifold equipped with a cone field is called a \textit{conal manifold}. A cone field is \textit{solid} if each convex cone $C_{M}(x)$ is solid.
%We say that a cone field admits a \textit{smooth section} if there is a smooth vector field $X$ such that $X(x)\in C_M(x)$ for all $x\in M$.

Let $x\in M$ and $\Phi \colon U \to V(\subset \mathbb{R}^r)$ be a smooth chart, where $U\subset M$ is an open set containing $x$. For any $y\in U$, denote by $C_{M}^{\Phi}(y)\subset \mathbb{R}^r$ the representation of the cone field in the chart, i.e., $d\Phi(y)(C_{M}(y))= C_{M}^{\Phi}(y)$. Following \cite{Lawson89},
a cone field $x \mapsto C_{M}(x)$ on $M$ is called \textit{upper semicontinuous} at $x\in M$ if given a smooth chart $\Phi \colon U \to \mathbb{R}^{r}$ at $x$ and a convex cone $C'$ surrounding $C_{M}^{\Phi}(x)$, there exists a neighborhood $W$ of $x$ such that $C_{M}^{\Phi}(y)\subset C'$ for all $y\in W$. The cone field is called \textit{lower semicontinuous} at $x$ if given any open set $S$ such that $S\cap C_{M}^{\Phi}(x)$ is nonempty, there exists a neighborhood $W$ of $x$ such that $S\cap C_{M}^{\Phi}(y)$ is nonempty for all $y\in W$. A cone field is \textit{continuous at $x$} if it is both lower and upper semicontinuous at $x$ and \textit{continuous} if it is continuous at every $x$.

Let $\theta : G\times M \to M$ be a left group action on $M$ such that each of the maps $\lambda_g \colon M\to M$ defined by $\lambda_g(x)=\theta(g,x)=g\cdot x$ forms a diffeomorphism of $M$.

\begin{defn}
A cone field $C_M$ is said to be \textit{G-invariant} or \textit{homogeneous} if $$d\lambda_g \vert_x (C_M(x))=C_M(g\cdot x) \text{~for all~} g\in G \text{~and~} x\in M.$$
\end{defn}

\begin{prop}\label{homogeneousconecontinuous}
	The homogeneous cone field on a homogeneous space is continuous.
\end{prop}
\begin{proof}
	See \cite[Proposition 4.6]{Lawson89}.
\end{proof}

\begin{defn}\label{conal-curve}
A continuous and piecewise continuously differentiable curve $t \mapsto \gamma(t)$ defined on $[t_0,t_1]$ into a conal manifold $M$ (equipped with a continuous cone field $C_M$) is called a \textit{conal curve} if $\gamma'(t)\in C_{M}(\gamma(t))$ whenever $t_0\leq t<t_1$, in which the derivative is the right hand derivative at those finitely many points where the derivative is not continuous. Moreover, $\gamma$ is called a \textit{strictly conal curve} if $\gamma$ is a conal curve and $\gamma'(t)\in \text{Int} C_{M}(\gamma(t))$ for $t_0\leq t<t_1$.
\end{defn}

\begin{prop}\label{groupkeepsorder}
	Suppose that a solid cone field $C_M$ on a homogeneous G-space $M$ is homogeneous. If a curve $\gamma \colon [t_0,t_1] \subset \mathbb{R} \to M$ is a conal curve, then $\lambda_g \circ \gamma$ is a conal curve for any $g\in G$, furthermore, if $\gamma$ is a strictly conal curve, so is $\lambda_g \circ \gamma$ for any $g\in G$.
\end{prop}
\begin{proof}
	Since $\gamma$ is a conal curve and $\frac{d}{ds}\gamma(s)\in C_{M}(\gamma(s))$, then $\frac{d}{ds}\lambda_g(\gamma(s))=d\lambda_g(\gamma(s))\frac{d}{ds}\gamma(s)\in C_{M}(\lambda_g(\gamma(s)))$ for all $t\in [t_0,t_1]$. Thus, $\lambda_g(\gamma(s))$ is a conal curve for any $g\in G$. Recall that the maps $\lambda_g$ is a diffeomorphism, then $\lambda_g \circ \gamma$ is a strictly conal curve if $\gamma$ is a strictly conal curve.
\end{proof}

\begin{defn}\label{def-order}
For two points $x,y \in M$, we write $x\leq_M y$ if there exists a conal curve $\alpha \colon [t_0,t_1] \subset \mathbb{R} \to M$ such that $\alpha(t_0)=x$ and $\alpha(t_1)=y$. We say that $x$ and $y$ are \textit{ordered}, if $x\leq_M y$ or $y\leq_M x$ holds.
This relation is an \textit{order} on $M$. In fact, it is always reflexive (i.e., $x\leq_M x$ for all $x\in M$) and transitive (i.e., $x\leq_M y$ and $y\leq_M z$ implies $x\leq_M z$). The order ``$\leq_M$" is referred to as the \textit{conal order}. We write $x<_M y$ if $x\leq_M y$ and $x\neq y$.
Furthermore, we say that $x$ and $y$ are \textit{unordered}, if neither $x\leq_M y$ nor $y\leq_M x$ holds.
We also write $x\ll_M y$ if there exists a so-called strictly conal curve $\gamma$ with $\gamma(t_0)=x$ and $\gamma(t_1)=y$. Clearly, the relation ``$\ll_M$" is always transitive. We say that $x$ and $y$ are \textit{strongly ordered}, if $x\ll_M y$ or $y\ll_M x$ holds.
A subset $U$ of $M$ is said to be \textit{ordered} (\textit{strongly ordered}, respectively), if any $x,y$ in $U$ are ordered (strongly ordered, respectively). $U$ is said to be \textit{unordered}, if any two points in $U$ are unordered.
\end{defn}

\subsection{Chronological sets}

Consider homogeneous space $M$ endowed with a homogeneous cone field $C_M$. For a given point $p\in M$, the \textit{causal future} $J^+(p)$, \textit{causal past} $J^-(p)$, \textit{chronological future} $I^+(p)$, \textit{chronological past} $I^-(p)$, \textit{chronological boundary} $\partial I^+(p)$ and $\partial I^-(p)$ of $p$ are defined as follows:
\begin{itemize}
	\item  $J^+(p)=\{q\in M \colon p\leq_M q\}$; \ \  $J^-(p)=\{q\in M \colon q\leq_M p\}$.
	
	\item  $I^+(p)=\{q\in M \colon p\ll_M q\}$; \ \  $I^-(p)=\{q\in M \colon q\ll_M p\}$.

    \item $\partial I^+(p)=\overline{I^+(p)}\backslash I^+(p)$; \ \ \ \ \ \  $\partial I^-(p)=\overline{I^-(p)}\backslash I^-(p)$.
\end{itemize}

\begin{rmk}
	Past and future sets have played an important role in singularity theory in general relativity as treated in Penrose \cite{Penrose72} or Hawking and Ellis \cite{Hawkingandellis73} via the allied concept of chronological boundaries. 
\end{rmk}

\begin{prop}\label{openness}
	For any $p\in M$, $I^{+}(p)$ and $I^-(p)$ are open sets.
\end{prop}	

\begin{proof}
	It follows from \cite[Proposition 2.8]{NW25}.
\end{proof}

\begin{prop}\label{homogeneous-prop}
	Suppose that a solid cone field $C_M$ on a homogeneous G-space $M$ is homogeneous, then for any $x,y\in M$ and $g_{x}^{y}\in G$, the following hold: $J^{+}(y)=g_{x}^{y}J^{+}(x)$, $J^{-}(y)=g_{x}^{y}J^{-}(x)$, $I^{+}(y)=g_{x}^{y}I^{+}(x)$, and $I^{-}(y)=g_{x}^{y}I^{-}(x)$.
\end{prop}
\begin{proof}
	The conclusion follows from Proposition \ref{groupkeepsorder}.
\end{proof}

\begin{defn}\label{continuity-for-chronological sets}
$I^+$ is said to be \textit{inner continuous} at $p\in M$ if each compact set $K\subset I^+(p)$, there exists a neighborhood $U(p)$ of $p$ such that $K\subset I^+(q)$ for each $q\in U(p)$.
$I^+$ is said to be \textit{outer continuous} at $p\in M$ if each compact set $K\subset M\backslash \overline{I^+(p)}$, there exists a neighborhood $U(p)$ of $p$ such that $K\subset M\backslash \overline{I^+(q)}$ for each $q\in U(p)$. $I^+$ is said to be \textit{continuous} at $p$ if it is both inner and outer continuous at $p$ and \textit{continuous} if it is continuous at every $p\in M$.
The relevant definitions of continuity for $I^-$ are similar.
\end{defn}

\begin{rmk}
	The definition of continuity of the chronological sets originates from Hawking and Sachs \cite[Section 1.6]{Hawkingandsachs74}, where they provided a crucial link between the continuity of the chronological sets and that of the time functions.
\end{rmk}

\begin{prop}\label{innercontinuous}
	For any $p\in M$, $I^{+}(p)$ and $I^-(p)$ are inner continuous.
\end{prop}	

\begin{proof}
	Suppose $K\subset I^{-}(p)$ is compact. For any $x\in I^-(p)$, i.e., $x\ll_M p$, there is a strictly conal curve $\gamma$ such that $\gamma (0)=x$ and $\gamma(1)=p$. Let $w=\gamma(\frac{1}{2})$, then $x\ll_M w\ll_M p$, i.e., $w\in I^-(p)$ and $x\in I^-(w)$. Since $I^-(w)$ is open, then $I^-(w)$ is an open neighborhood of $x$. Thus, $\{I^-(w): w\in I^-(p)\}$ is an open covering of $K$. Since $K$ is a compact set, we choose $w_1, w_2, \cdots, w_n$ to determine a finite subcovering. On the other hand, $w_i\in I^-(p)$ implies that $p\in I^+(w_i)$, $i=1,2,\cdots, n$. So, $U=\bigcap^{n}_{i=1}I^+(w_i)$ is an open neighborhood of $p$. For any $u\in U$, $u\in I^+(w_i)$, $i=1,2,\cdots, n$. Then $w_i\in I^-(u)$. For any $y\in I^-(w_i)$, since $y\ll_M w_i$ and $w_i\ll_M u$, then $y\ll_M u$, i.e., $y\in I^-(u)$. So, $I^-(w_i)\subset I^-(u)$ for $i=1,2,\cdots,n$. Thus, $K\subset \bigcup ^{n}_{i=1}I^-(w_i)\subset I^-(u)$. So, we have proved that $I^{-}(p)$ is inner continuous.
	
	A similar argument is used for $I^+(p)$.
\end{proof}

\begin{defn}\label{quasi-closed}
	The order ``$\leq_M$" is said to be \textit{quasi-closed} if $x\leq_M y$ whenever $x_n\to x$ and $y_n\to y$ as $n\to \infty$ and $x_n\ll_M y_n$ for all $n$.
\end{defn}

\begin{prop}\label{quasi-closed-prop}
	If the conal order ``$\leq_M$" is quasi-closed, then $\partial I^{+}(p)\subset J^{+}(p)\backslash I^{+}(p)$ and $\partial I^{-}(p)\subset J^{-}(p)\backslash I^{-}(p)$.
\end{prop}	

\begin{proof}
	If $q\in \partial I^{+}(p)$, then there is a sequence $\{p_n\}\subset I^{+}(p)$ such that $p_n \to q$. The quasi-closedness of the relation ``$\leq_M$" implies that $p\leq_M q$. Thus, we have $q\in J^{+}(p)$ and hence $\partial I^{+}(p)\subset J^{+}(p)$. Moreover, \cite[Proposition 3.4]{Penrose72} implies that $\partial I^{+}(p)\subset J^{+}(p)\backslash I^{+}(p)$. The case for $\partial I^{-}(p)$ can be proved analogously.
\end{proof}

	\begin{prop}\label{chronological boundary}
		Suppose that a solid cone field $C_M$ on a homogeneous G-space $M$ is homogeneous. Then, for any $p\in M$, we have $p\in \overline{I^{+}(p)}\cap \overline{I^{-}(p)}$. Moreover,
		\begin{enumerate}[{\rm (i)}]
			\item if $p\in I^{+}(p)\cup I^{-}(p)$, then $p\in I^{+}(p)\cap I^{-}(p)$;
			
			\item if $p\in \partial I^{+}(p)\cup \partial I^{-}(p)$, then $p\in \partial I^{+}(p)\cap \partial I^{-}(p)$.
		\end{enumerate}
	\end{prop}	
	\begin{proof}
		Since cone field $C_M$ is solid, \cite[Proposition 4.6]{Lawson89} implies that there exists a smooth vector field $X$ such that $X(p)\in \text{Int} C_M(p)$ for any $p\in M$. Therefore, for each point $p\in M$, there exist $\epsilon>0$ and a smooth curve $\gamma \colon (-\epsilon,\epsilon)\to M$ that is an integral curve of $X$ starting at $p$ (i.e., $\gamma(0)=p$). Consequently, $\gamma$ is a strictly conal curve. This fact enables us to assume that $I^{+}(p)\neq \emptyset$ and $I^{-}(p)\neq \emptyset$ for any $p\in M$.
		If $q_1\in I^{+}(p)$, then there exists a strictly conal curve $\gamma_1 \colon [0,1] \subset \mathbb{R} \to M$ such that $\gamma_1(0)=p$ and $\gamma_1(1)=q_1$. Clearly, $\gamma_1(s)\in I^{+}(p)$ for all $s\in (0,1)$ and $\lim_{s\to 0}\gamma_1(s)=\gamma_1(0)=p$. So, $p\in \overline{I^{+}(p)}$. Similarly, if $q_2\in I^{-}(p)$, then there exists a strictly conal curve $\gamma_2 \colon [0,1] \subset \mathbb{R} \to M$ such that $\gamma_2(0)=q_2$ and $\gamma_2(1)=p$. Furthermore, $\gamma_2(s)\in I^{-}(p)$ for all $s\in (0,1)$ and $\lim_{s\to 1}\gamma_2(s)=\gamma_2(1)=p$. Therefore, $p\in \overline{I^{-}(p)}$. Hence, $p\in \overline{I^{+}(p)}\cap \overline{I^{-}(p)}$.
		
		For (i). If $p\in I^{+}(p)$ (or $p\in I^{-}(p)$), then $p\ll_M p$. And hence, $p\in I^{-}(p)$ (resp., $p\in I^{+}(p)$). Thus, $p\in I^{+}(p)\cap I^{-}(p)$.
		
		For (ii). If $p\in \partial I^{+}(p)$, then $p\notin I^{+}(p)$. So, $p\notin I^{-}(p)$ by (i). Furthermore, the fact that $p\in \overline{I^{+}(p)}\cap \overline{I^{-}(p)}$ entails that $p\in \partial I^{-}(p)$. By a similar argument, one can prove that if $p\in \partial I^{-}(p)$, then $p\in \partial I^{+}(p)$. Thus, we obtain the conclusion.
	\end{proof}

\section{Differentially positive systems and main results}\label{DP}

Following the geometric framework established in Section \ref{s2}, we introduce the differentially positive systems on homogeneous spaces and present the main results of the paper.

\subsection{Differentially positive systems}\label{SDP}

Now, we consider a system $\Sigma$ generated by a continuously differentiable vector field $f$ on homogeneous space $M$ equipped with a homogeneous cone field $C_M$. The induced flow by system $\Sigma$ is denoted by $\varphi$, which is a map $\varphi \colon \mathbb{R} \times M \to M$, given by $(t, x) \mapsto \varphi_t(x)$. For brevity, we also denote the flow by $\varphi_t$. We write $d\varphi_t(x)$ as the tangent map from $T_{x}M$ to $T_{\varphi_{t}(x)}M$. Let $x\in M$, the \textit{positive semiorbit} (resp., \textit{negative semiorbit}) of $x$, be denoted by $O^+(x)=\{\varphi_t(x) \colon t\geq 0\}$ (resp.,  $O^-(x)=\{\varphi_t(x)\colon t\leq 0\}$). The \textit{orbit} of $x$ is denoted by $O(x)=O^+(x)\cup O^-(x)$. The $\omega$-limit set $\omega(x)$ of $x$ is defined by $\omega(x)=\cap_{s\geq 0}\overline{\cup_{t\geq s}\varphi_t(x)}$. A point $z\in \omega(x)$ if and only if there exists a sequence $\{t_i\}$, $t_i\to \infty$, such that $\varphi_{t_i}(x)\to z$ as $i\to \infty$. If $O^+(x)$ is precompact (i.e., the closure of $O^{+}(x)$ is compact), then $\omega(x)$ is nonempty, compact, connected, and invariant (i.e., $\varphi_t(\omega(x))=\omega(x)$ for any $t\in \mathbb{R}$).

A point $x\in M$ is called a \textit{recurrent point} if $x\in \omega(x)$. Denote by $\mathcal{R}(\varphi)$ all recurrent points of $\varphi_t$ in $M$. The closure of $\mathcal{R}(\varphi)$ in $M$ is called \textit{Birkhoff center}, labeled by $\mathcal{B}(\varphi)$, i.e.,
\begin{equation*}
	\mathcal{B}(\varphi)=\overline{\{x\in M \colon x\in \omega(x)\}}.
\end{equation*}

\begin{defn}\label{differential positive}
	The system $\Sigma$ is said to be \textit{strongly differentially positive} (SDP) with respect to $C_{M}$ if
	\begin{equation*}
		d\varphi_t(x)\{ C_{M}(x)\backslash \{ 0\} \} \subseteq \text{Int} C_{M}(\varphi_t(x)), \ \ \forall x\in M, \ \ \forall t>0.
	\end{equation*}
\end{defn}

\vskip 3mm

The following two propositions indicate that the strongly differentially positive system $\Sigma$ naturally keeps the conal order.

	\begin{prop}\label{monotone property}
		Assume that system $\Sigma$ is {\rm SDP}. 
		\begin{enumerate}[{\rm (i)}]
			\item	If $x<_M y$, then $\varphi_t(x)\ll_M \varphi_t(y)$ for all $t>0$.
			\item   If $x\leq_M x$ with $x\in \gamma$, then $\varphi_t(x)\ll_M \varphi_t(x)$ for all $t>0$, where $\gamma$ is a nontrivial closed conal curves.
		\end{enumerate}
	\end{prop}
	\begin{proof}
		If $x<_M y$, there exists a conal curve $\gamma(s)$ such that $\gamma(0)=x$ and $\gamma(1)=y$. Since $\Sigma$ is SDP and $\frac{d}{ds}\gamma(s)\in C_{M}(\gamma(s))\backslash \{ 0\}$, then $\frac{d}{ds}\varphi_t(\gamma(s))=d\varphi_t(\gamma(s))\frac{d}{ds}\gamma(s)\in \text{Int} C_{M}(\varphi_t(\gamma(s)))$ for $t>0$. So, $\varphi_t(\gamma(s))$ is a strictly conal curve and $\varphi_t(x)\ll_M \varphi_t(y)$ for all $t>0$.
		
		If $x$ belongs to a nontrivial closed conal curve $\gamma$ such that $\gamma(0)=\gamma(1)=x$, then $\varphi_t(\gamma)$ is a closed strictly conal curve and $\varphi_t(x)\in \varphi_t(\gamma)$ for all $t>0$.
	\end{proof}
	
	\begin{rmk}\label{closed-conal-curve}
		As a matter of fact, if $M$ contains nontrivial closed conal curves, then case (ii) of Proposition \ref{monotone property} will occur, which implies that $\varphi_t(x)\in I^{+}(\varphi_t(x))$ for all $t>0$ whenever $x$ belongs to a nontrivial closed conal curve. This phenomenon cannot occur in classical strongly monotone systems.
	\end{rmk}

\begin{prop}\label{strong flow open relation}
	Assume that system $\Sigma$ is {\rm SDP}. If $x<_M y$, then for $t_0>0$, there exist neighborhoods $U$ of $x$, $V$ of $y$ such that $\varphi_{t}(U)\ll_M \varphi_{t}(V)$ for any $t\geq t_0$.
\end{prop}

\begin{proof}
	Since $x<_M y$, then $\varphi_{t_0}(x)\ll_M \varphi_{t_0}(y)$ for $t_0>0$. One can take neighborhoods $\bar{U}$ of $\varphi_{t_0}(x)$ and $\bar{V}$ of $\varphi_{t_0}(y)$ such that $\bar{U}\ll_M\bar{V}$ by Proposition \ref{openness}. By the continuity of $\varphi_{t_0}$, there are neighborhoods $U$ of $x$, $V$ of $y$ such that $\varphi_{t_0}(U)\subset \bar{U}$ and $\varphi_{t_0}(V)\subset \bar{V}$.
\end{proof}

\subsection{Main results}\label{mr}

Throughout this paper, we always make the following assumptions:
Homogeneous space $M=G/H$ is endowed with a $G$-invariant Riemannian metric $(\cdot,\cdot)^{M}$ (which induces a Riemannian distance $d$ on $M$) and Lie group $G$ is endowed with a left-invariant Riemannian metric $(\cdot,\cdot)^{G}$ (which induces a Riemannian distance $d_G$ on $G$).

\begin{enumerate}
	\item[{\bf (H1)}] The cone field $C_M$ is a solid homogeneous cone field. 
	%($ I^{\pm}(x)=g_{x_{t}}^{x} I^{\pm}(x_{t})$)
	\item[{\bf (H2)}] The conal order ``$\leq_M$" is quasi-closed. %($\varphi_{t}(\overline{I^{\pm}(p)}\backslash\{p\})\subset I^{\pm}(\varphi_{t}p)$ for $t>0$)
	%\item[{\bf (H3)}] The Riemannian metric on $M$ is invariant.  %($d(x,y)=d(g_{a}^{b}x,g_{a}^{b}y)$)
	\item[{\bf (H3)}] The chronological sets $I^+$ and $I^-$ are continuous.
	\item[{\bf (H4)}] The system $\Sigma$ is {\rm SDP}, furthermore, each orbit is well-defined for all $t\geq 0$ and possesses a compact closure.
\end{enumerate}

We have the following theorem:

\begin{thm}\label{Thm_2.1}
	Assume that  {\rm (H1)-(H4)} hold. Let $B$ be a connected component of Birkhoff center $\mathcal{B}(\varphi)$. Then $B$ is either
	\begin{enumerate}[{\rm(i)}]
		\item strongly ordered, i.e., any two points in $B$ are strongly ordered; or
		\item unordered, i.e., any two points in $B$ are unordered.
	\end{enumerate}
\end{thm}

This theorem will be proved in Section \ref{Limit dich}. A direct consequence of Theorem \ref{Thm_2.1} is a characterization of the order structure of the supports of invariant measures on $M$. Recall that a probability measure $\mu$ is said to be invariant if $\mu \circ \varphi=\mu$. The support of $\mu$, denoted by $\text{supp} (\mu)$, is defined as the complement in $M$ of the union of all open sets $U$ satisfying $\mu(U)=0$. Since $\text{supp}(\mu)\subset \mathcal{B}(\varphi)$ (see, e.g., \cite[p.28]{M12} or \cite[Appendix A]{J20}), Theorem \ref{Thm_2.1} directly yields the following result:

\vskip 3mm

\begin{cor}\label{Thm_2.2}
	Assume that  {\rm (H1)-(H4)} hold. Let $B\subset\text{supp} (\mu)$ be any connected component of the support of an invariant measure $\mu$. Then $B$ is either strongly ordered or unordered.
\end{cor}

\vskip 3mm

\begin{rmk}
	Forni and Sepulchre \cite{ForniandSepulchre16} first investigated the dynamics of differentially positive systems on a Riemannian manifold. They established a dichotomy for the limit sets of differentially positive systems (see \cite[Theorem 4]{ForniandSepulchre16}). Recently,	the present authors \cite{NW25,NWZ26} characterized the generic behavior of differentially positive systems, mainly from the perspective of $\omega$-limit sets for individual orbits. 
	However, we point out that  the recurrence and statistical behavior cannot solely rely on the	analysis of the limit sets of individual orbits. In light of our results, we have presented feasible and effective information for the recurrence of differentially positive systems.
\end{rmk}

\begin{rmk}
	Needless to say, differentially positive systems can be regarded as a natural generalization of the so-called classical monotone systems from flat spaces to nonlinear manifolds. On flat spaces, Theorem \ref{Thm_2.1}, as well as Corollary \ref{Thm_2.2}, reduces to a celebrated theorem of Hirsch (see \cite[Theorem 1.6]{H99} and \cite{Mi00}), which states that any connected component $B$ of the Birkhoff center (or more generally, any attractor-free set $B$) either is unordered, or is a stationary p-arc. This theorem has been wildly applied to investigate the asymptotic behavior of stochastic approximation, as well as stochastic stability of differential equations (see \cite{B00,BF21,BH99,JWS26} and references therein). For differentially positive systems, we are {\it led to conjecture} that, if the connected component of the Birkhoff center	on a homogeneous space is strongly ordered, then it must be either a fixed point or a periodic orbit. 
\end{rmk}

\noindent\section{Intersection principle for chronological boundary}\label{s3}
	
   	In this section, we aim to establish the intersection principle by analyzing the chronological boundary of any recurrent point, which turns out to be an important tool to prove Theorem \ref{Thm_2.1}. 
    	
   	Throughout this section, all the assumptions in Subsection \ref{mr} are assumed to hold.

	\begin{lem}\label{endpoint}
		{\rm(Intersection Principle).}
		Let $x\in \mathcal{R}(\varphi)$. Then the following two assertions hold:
        \begin{enumerate}[{\rm (i)}]
			\item If $x\in \partial I^{+}(x)\cup \partial I^{-}(x)$, then $\mathcal{B}(\varphi)\cap (\partial  I^{+}(x)\cup \partial  I^{-}(x))=\{x\}$;
            
            \item If $x\in I^{+}(x)\cup I^{-}(x)$, then $\mathcal{B}(\varphi)\cap (\partial  I^{+}(x)\cup \partial  I^{-}(x))=\emptyset$.
        \end{enumerate}
	\end{lem}

    \begin{rmk}
    This lemma reveals that on the chronological boundary of any recurrent point $x$, there are no other points of Birkhoff center, except possibly for $x$ itself. 
    The classification in this lemma depends on whether $x$ belongs to a nontrivial closed conal curve or not. Moreover, Proposition \ref{chronological boundary} gives the the location of $x$, $x\in \overline{I^{+}(x)}\cap \overline{I^{-}(x)}$, and ensures the classification is reasonable.
    \end{rmk}

   	The proof of Lemma \ref{endpoint} is presented in Subsection \ref{s3.4}. To achieve this, we analyze the chronological boundary of any recurrent point $x$. First, in Subsection \ref{s3.1}, we provide some preliminary results and present properties regarding the order structure of $\omega$-limit sets. Then, Subsection \ref{s3.2} introduces a soft intersection principle (Lemma \ref{sep_n}), demonstrating that $x$ is the only possible recurrent point on $\partial I^{+}(x)\cup \partial I^{-}(x)$. Next, in Subsection \ref{s3.3}, we establish a technical lemma “limit-set dichotomy for recurrent points" (Lemma \ref{limit_set_dichotomy}), which proves crucial for deriving the intersection principle.

\subsection{Order structure of $\omega$-limit sets}\label{s3.1}

\begin{prop}\label{pseudoorder-orbit}
	If $x\in \mathcal{R}(\varphi)$ and there is some $T>0$ such that either $x<_M \varphi_T(x)$ or $\varphi_T(x)<_M x$, then $\omega(x)$ is ordered, furthermore, for any $y_1, y_2 \in \omega(x)$, both $y_1\leq_M y_2$ and $y_2\leq_M y_1$ hold. 
\end{prop}
	
\begin{proof}
	Firstly, if there is some $T>0$ such that $x<_M \varphi_T(x)$, we {\it claim that there exists $y\in O^{+}(x)$ such that $y\ll_M \varphi_{t}(y)$ for all $t>0$}. If the claim holds, we could obtain the proposition immediately. In fact, the claim implies that $O^{+}(y)$ is strongly ordered, i.e., for any $0<t_1<t_2$, $\varphi_{t_1}(y)\ll_M \varphi_{t_2}(y)$. Let $y_1, y_2$ be two distinct points in $\omega(x)$. So, $y_1, y_2 \in \omega(y)$ and there exist two sequences $t_k \to \infty$ and $t_l \to \infty$ such that $\varphi_{t_k}(y)\to y_1$ as $k\to \infty$ and $\varphi_{t_l}(y)\to y_2$ as $l\to \infty$. Then for each $k$, there is a $l(k)$ such that $t_k<t_{l(k)}$ and $\varphi_{t_k}(y)\ll_M \varphi_{t_{l(k)}}(y)$. Thus, $y_1\leq_M y_2$ by the quasi-closedness of the order. A similar argument shows that $y_2\leq_M y_1$.

	Now, it suffces to prove the claim.	Since the system is strongly differentially positive, the fact that $x<_M \varphi_T(x)$ implies that $\varphi_{t}(x)\ll_M \varphi_{t}(\varphi_T(x))$ for any $t>0$. Without loss of generality, we write $x\ll_M \varphi_T(x)$. Let $T_1=\inf \{0 < t \leq T \colon x\ll_M \varphi_{\tau}(x) \text{ for all } \tau\in [t,T] \}$. The openness of the relation ``$\ll_M$" implies that $T_1<T$. Then $x\leq_M \varphi_{T_1}(x)$ and $x\ll_M \varphi_{\tau}(x)$ for all $\tau\in (T_1,T]$.
		
	If $x\neq \varphi_{T_1}(x)$, we will get a contradiction. In fact, there exist $t_0>0$ and neighborhoods $U_1$ of $x$, $V_1$ of $\varphi_{T_1}(x)$ such that $U_1 \cap V_1=\emptyset$ and $\varphi_{t}(U_1) \ll_M \varphi_{t}(V_1)$ for all $t\geq t_0$. Choose $\epsilon_1>0$ so small such that $B_{\epsilon_1}(\varphi_{T_1}(x))\subset V_1$, where $B_{\epsilon_1}(\varphi_{T_1}(x))$ is the ball centered at $\varphi_{T_1}(x)$ with the radius $\epsilon_1$. For such $\epsilon_1$ choose some $\delta>0$ such that $\varphi_{T_1}(B_{\delta}(x))\subset B_{\frac{\epsilon_1}{3}}(\varphi_{T_1}(x))$, and hence there exists some $l_1>0$ such that $\varphi_{T_1+s}(x)\in B_{\frac{\epsilon_1}{3}}(\varphi_{T_1}(x))$ for any $s\in [-l_1,l_1]$. So, for every $s\in [-l_1,l_1]$, one can find some $\delta(s)\in (0,\delta)$ such that $\varphi_{T_1+s}(B_{\delta(s)}(x))\subset B_{\epsilon_1}(\varphi_{T_1}(x))$. Furthermore, it follows from the compactness of $[-l_1,l_1]$ that $\inf_{s\in [-l_1,l_1]}\delta(s)=\delta_1 >0$. As a consequence, 
	\begin{equation}\label{pseudoorder-equ-1}
		\varphi_{T_1+s}(B_{\delta_1}(x))\subset B_{\epsilon_1}(\varphi_{T_1}(x))\subset V_1, \text{ for all } s\in [-l_1,l_1].
	\end{equation}
	One can also let $\delta_1$ be small, if necessary, such that $B_{\delta_1}(x)\subset U_1$. Hence, together with the fact that $\varphi_{t}(U_1) \ll_M \varphi_{t}(V_1)$ for all $t\geq t_0$, this implies that
	\begin{equation}\label{pseudoorder-equ-2}
		\varphi_{t}(B_{\delta_1}(x))\ll_M \varphi_{t}(B_{\epsilon_1}(\varphi_{T_1}(x))), \text{ for all } t\geq t_0.
	\end{equation}
	Choose $\epsilon>0$ so small such that $\varphi_{-\epsilon}(x)\in B_{\delta_1}(x)$. Then, \eqref{pseudoorder-equ-1} implies that $\varphi_{T_1+s}(\varphi_{-\epsilon}(x))\subset B_{\epsilon_1}(\varphi_{T_1}(x))\subset V_1$, for all $s\in [-l_1,l_1]$. Since $x\in \mathcal{R}(\varphi)$, there exists a sequence $\tau_n \to \infty$ such that $\varphi_{\tau_n}(x)\to x$ as $n\to \infty$. Together with \eqref{pseudoorder-equ-2}, this yields that
	\[
	\varphi_{\tau_n}(\varphi_{-\epsilon}(x))\ll_M \varphi_{\tau_n}(\varphi_{T_1+s}(\varphi_{-\epsilon}(x)))
	\]
	for all $s\in [-l_1,l_1]$ and $n$ sufficiently large. The quasi-closedness of the order implies that $\varphi_{-\epsilon}(x)\leq_M \varphi_{T_1+s}(\varphi_{-\epsilon}(x))$ as $n\to \infty$. Since $\varphi_{-\epsilon}(x)\subset U_1$ and $\varphi_{T_1+s}(\varphi_{-\epsilon}(x))\subset V_1$, then $\varphi_{-\epsilon}(x)\neq \varphi_{T_1+s}(\varphi_{-\epsilon}(x))$. Furthermore, the strongly differential positivity of the system implies that $x\ll_M \varphi_{T_1+s}(x)$ for any $s\in [-l_1,l_1]$, which contradicts the definition of $T_1$. So, we have proved that $x= \varphi_{T_1}(x)$. And hence $\varphi_{T_1}(x)\ll_M \varphi_{\tau}(x)$ for all $\tau\in (T_1,T]$.
		
	Let $y=\varphi_{T_1}(x)$. Then we have the following fact, i.e., there exists a time, denoted by $\tilde{T}$, such that $y\ll_M \varphi_{t}(y)$ for all $t\in (0,\tilde{T}]$. Let $T_2=\sup \{t \geq \tilde{T} \colon y\ll_M \varphi_{\tau}(y) \text{ for all } \tau\in [\tilde{T},t] \}$. If $T_2=+\infty$, we obtain the claim. If $T_2<+\infty$ with $y=\varphi_{T_2}(y)$, then the orbit passing though $y$ is a periodic orbit and each point on the orbit has the order relation with $y$. So, we obtain the claim.
		
	We only need to consider the case $T_2<+\infty$ with $y\neq \varphi_{T_2}(y)$. In this case, we will get a contradiction. In fact, the definition of $T_2$ implies that $y<_M \varphi_{T_2}(y)$. As a consequence, there exist $\tilde{t}_0>0$ and neighborhoods $U_2$ of $y$, $V_2$ of $\varphi_{T_2}(y)$ such that $U_2 \cap V_2 =\emptyset$ and $\varphi_{t}(U_2) \ll_M \varphi_{t}(V_2)$ for all $t\geq \tilde{t}_0$. Choose $\epsilon_2>0$ so small that $B_{\epsilon_2}(\varphi_{T_2}(y))\subset V_2$. It then follows from a similar argument for $T_1$ above that there are constants $l_2,\delta_2 >0$ such that
	\begin{equation}\label{pseudoorder-equ-3}
		\varphi_{T_2+s}(B_{\delta_2}(y))\subset B_{\epsilon_2}(\varphi_{T_2}(y))\subset V_2, \text{ for all } s\in [-l_2,l_2].
	\end{equation}
	By choose such $\delta_2>0$ small, if necessary, we assume that $B_{\delta_2}(y)\subset U_2$. Then 
	\begin{equation}\label{pseudoorder-equ-4}
		\varphi_{t}(B_{\delta_2}(y))\ll_M \varphi_{t}(B_{\epsilon_2}(\varphi_{T_2}(y))), \text{ for all } t\geq \tilde{t}_0.
	\end{equation}
	Choose $\tilde{\epsilon}>0$ so small such that $\varphi_{-\tilde{\epsilon}}(y)\in B_{\delta_2}(y)$. By virtue of \eqref{pseudoorder-equ-3}, we obtain the fact that $\varphi_{T_2+s}(\varphi_{-\tilde{\epsilon}}(y))\subset B_{\epsilon_2}(\varphi_{T_2}(y))\subset V_2$, for all $s\in [-l_2,l_2]$. Recall that $x\in \mathcal{R}(\varphi)$ and $y=\varphi_{T_1}(x)$. Then $y\in \mathcal{R}(\varphi)$, i.e., there exists a sequence $\tilde{\tau}_n \to \infty$ such that $\varphi_{\tilde{\tau}_n}(y)\to y$ as $n\to \infty$.  Together with \eqref{pseudoorder-equ-4}, this implies that
	\[
	\varphi_{\tilde{\tau}_n}(\varphi_{-\tilde{\epsilon}}(y))\ll_M \varphi_{\tilde{\tau}_n}(\varphi_{T_2+s}(\varphi_{-\tilde{\epsilon}}(y)))
	\]
	for all $s\in [-l_2,l_2]$ and $n$ sufficiently large. By letting $n\to +\infty$ again, we have that $\varphi_{-\tilde{\epsilon}}(y)<_M \varphi_{T_2+s}(\varphi_{-\tilde{\epsilon}}(y))$. Moreover, the strongly differential positivity of the system implies that $y\ll_M \varphi_{T_2+s}(y)$ for any $s\in [-l_2,l_2]$, which contradicts the definition of $T_2$. Thus, we obtain the claim.
	
	Secondly, for the case that there is some $T>0$ such that $\varphi_T(x)<_M x$, the proof follows the similar arguments. Thus, we obtain the proposition.
\end{proof}

\begin{prop}\label{notin}
	Let $x,z\in \mathcal{R}(\varphi)$. If $z\in \partial I^{+}(x)\cup \partial I^{-}(x)$, then $x\notin \omega(z)$.
\end{prop}

\begin{proof}
	Suppose $x\in \omega(z)$, then for any $t_0>0$, it has $\varphi_{t_0}(x)\in \omega(z)$. If $z\in \partial I^+(x)$ (the case $z\in \partial I^-(x)$ is similar), Proposition \ref{quasi-closed-prop} implies $\varphi_{t_0}(x)\ll_M \varphi_{t_0}(z)$ which entails that there is a neighborhood $U$ of $\varphi_{t_0}(x)$ such that $y\ll_M \varphi_{t_0}(z)$ for any $y\in U$. Since $\varphi_{t_0}(x)\in \omega(z)$, there is a $t_1>0$ such that $t_1\neq t_0$ and $\varphi_{t_1}(z)\in U$. And hence, $\varphi_{t_1}(z)\ll_M \varphi_{t_0}(z)$. Then Proposition \ref{pseudoorder-orbit} shows that $\omega(z)$ is ordered. Noticed $x\in \omega(z)$ and $z\in \omega(z)$, the invariant of the set $\omega(z)$ entails that there exists an $\epsilon>0$ such that $\varphi_{-\epsilon}(x)<_M \varphi_{-\epsilon}(z)$. Moreover, by strong monotonicity, it has $x\ll_M z$, i.e., $z\in I^{+}(x)$, a contradiction to $z\in \partial I^+(x)$.
\end{proof}

\begin{prop}\label{ou}
	Let $x\in \mathcal{R}(\varphi)$. Then $\omega(x)$ is either strongly ordered or unordered.
\end{prop}

\begin{proof}
	If $\omega(x)$ is not unordered, then there are two points $a,b$ in $\omega(x)$ such that $a<_M b$. By strong monotonicity, there is $\tau>0$ such that $\varphi_{\tau}(a)\ll_M \varphi_{\tau}(b)$. Recall $a\in\omega(x)$ and $b\in\omega(x)$, there are two points $\varphi_{t_1}(x)$ and $\varphi_{t_2}(x)$ in $O^{+}(x)$ closed to $\varphi_{\tau}(a)$ and $\varphi_{\tau}(b)$ respectively such that $\varphi_{t_1}(x)\ll_M \varphi_{t_2}(x)$. Then Proposition \ref{pseudoorder-orbit} shows that $\omega(x)$ is ordered and for any two distinct points $y_1, y_2 \in \omega(x)$, both $y_1<_M y_2$ and $y_2<_M y_1$ hold. The invariant of the set $\omega(x)$ entails that $\varphi_{-t}(y_1), \varphi_{-t}(y_2)\in \omega(x)$ whenever $y_1, y_2 \in \omega(x)$ and $t>0$. Furthermore, both $\varphi_{-t}(y_1)<_M \varphi_{-t}(y_2)$ and $\varphi_{-t}(y_2)<_M \varphi_{-t}(y_1)$ hold. So, the strong monotonicity implies that both $y_1\ll_M y_2$ and $y_2\ll_M y_1$ hold for any $y_1, y_2 \in \omega(x)$. Therefore, $\omega(z)$ is strongly ordered, which completes the proof.
\end{proof}

    \subsection{A soft version of the intersection principle}\label{s3.2}	

    In this subsection, we are focus on the order relationship of recurrent points and present a soft version of the intersection principle, that is,

\begin{lem}\label{sep_n}
	Let $x\in \mathcal{R}(\varphi)$. Then the following two assertions hold:
    \begin{enumerate}[{\rm (i)}]
			\item If $x\in \partial I^{+}(x)\cup \partial I^{-}(x)$, then $\mathcal{R}(\varphi)\cap (\partial  I^{+}(x)\cup \partial  I^{-}(x))=\{x\}$;
            
            \item If $x\in I^{+}(x)\cup I^{-}(x)$, then $\mathcal{R}(\varphi)\cap (\partial  I^{+}(x)\cup \partial  I^{-}(x))=\emptyset$.
    \end{enumerate}
    
\end{lem}

%In the following, we just prove that $\mathcal{R}(\varphi)\cap \partial I^{+}(x)=\{x\}$ if $x\in \partial I^{+}(x)\cup \partial I^{-}(x)$, and $\mathcal{R}(\varphi)\cap \partial  I^{+}(x)=\emptyset$ if $x\in I^{+}(x)\cup I^{-}(x)$, as the remainder can be obtained by a similar argument.

Before proving the Lemma \ref{sep_n}, we label some useful notion to express clearly. In the rest part of this subsection, let $x, z\in \mathcal{R}(\varphi)$ with $x\neq z$ and $z\in \partial  I^{+}(x)$. For each $y\in M$, we label $$W_{y}=\omega(z)\cap  \overline{I^{+}(y)} ,$$ and for $x\in \mathcal{R}(\varphi)$, $$x_{t}\triangleq \varphi_{t}(x),\ \ \ W_{x}^{t}\triangleq g_{x_{t}}^{x}\varphi_{t}(W_{x}),$$
where $g_{x_{t}}^{x}\in G$ is the group element that maps $x_{t}$ to $x$.
Clearly, $z\in W_x$. Note that $x\notin\omega(z)$ by Proposition \ref{notin}, one can find $\eta_2>\eta_1>0$ and a closed annulus domain $A=\{v\in M:\eta_1\le d(v,x)\le \eta_2\}$ such that $\omega(z)\subset {\rm Int}A$. We label $\partial  I^{+}(x)$ restricted to $A$ by $\partial I^{+}_{A}(x)$, i.e., $$\partial  I^{+}_{A}(x)=\partial I^{+}(x)\cap A.$$ Clearly, $\partial  I^{+}_{A}(x)\neq\emptyset$ as $z\in \partial I^{+}(x)$ and $z\in\omega(z)(\subset {\rm Int}A)$. Moreover, whether $x\in \partial I^{+}(x)\cup \partial I^{-}(x)$ or $x\in I^{+}(x)\cup I^{-}(x)$, it has $x\notin\partial  I^{+}_{A}(x)$.

To be more precise, some notion about the distance between two sets are necessary. For any two subsets $A,B\subset M$, the \textit{Hausdorff distance} between $A$ and $B$ is defined as $$d_{\mathcal{H}}(A,B)=\max\{\sup_{a\in A} d(a,B),\sup_{b\in B} d(b,A)\},$$
where $d(a,B)=\inf_{b\in B}d(a,b)$ and $d(b,A)=\inf_{a\in A}d(a,b)$.
While, the \textit{separation index} between $A$ and $B$ is defined as $${\rm \underline{dist}}(A,B)=\inf_{x\in A,y\in B} d(x,y).$$
Clearly, ${\rm\underline{dist}}(A,B)>0$ if and only if $A\cap B=\emptyset$ when $A,B$ are compact. In particular, $d(x,B)={\rm \underline{dist}}(\{x\},B)$.
Some useful relationships between $d_{\mathcal{H}}$ and ${\rm\underline{dist}}$ are listed below:
\begin{align}
	&{\rm\underline{dist}}(A,B)\leq{\rm\underline{dist}}(A,C)+d_{\mathcal{H}}(B,C),\label{dis_2}
\end{align}
	where $A,B,C$ are closed sets in $ M$. 

\begin{prop}\label{R(x_t)}
    There exists $\delta>0$ such that for any $t>1$, 
    \begin{equation}\label{w}
        {\rm\underline{dist}}(W_{x}^{t},\partial  I^{+}_{A}(x))\geq\delta.
    \end{equation} 
\end{prop}

\begin{proof}
    For any $x_t$, label $$B_{r}(x_t)=\{v: d(v,x_t)< r\}\text{ and }I^{+}_{r}(x_t)=I^{+}(x_t)\cap B_{r}(x_t).$$
    Choose $r>0$ large enough so that $\omega(z)\subset \cap_{t\geq0} B_{r}(x_t)$. Thus, $\omega(z)\subset B_{r}(x_t)$ for any $x_t$. Recalling $W_{x_t}=\omega(z)\cap\overline{I^{+}(x_t)}$, it has $W_{x_t}\subset B_{r}(x_t)\cap\overline{I^{+}(x_t)}$, i.e., $W_{x_t}\subset \overline{I^{+}_{r}(x_t)}$, for any $t>0$. For any $t>1$, label $$R(x_t)=g_{x_{t}}^{x}\varphi_1 (I^{+}_{r}(x_{t-1})).$$ Then, $\overline{R(x_t)}=g_{x_{t}}^{x}\varphi_1 (\overline{I^{+}_{r}(x_{t-1})})$ for any $t>1$. Moreover, it has $W_x^t\subset \overline{R(x_t)}$ as 
	$\varphi_{t} (W_x)\subset \varphi_1 (W_{x_{t-1}})\subset \varphi_1 (\overline{I^{+}_{r}(x_{t-1})})$ for any $t>1$. This implies 
    \begin{equation*}\label{d>d}
		{\rm \underline{dist}}(W_{x}^{t},\partial  I^{+}_{A}(x))\geq {\rm\underline{dist}}(\overline{R(x_t)}, \partial  I^{+}_{A}(x)) \text{ for any } t>1.
	\end{equation*}
    Thus to prove \eqref{w}, it suffices to prove that there exists $\delta>0$ such that 
    \begin{equation}\label{>delta}
        {\rm\underline{dist}}(\overline{R(x_t)},\partial  I^{+}_{A}(x))>\delta \text{ for any }t>1.
    \end{equation}

    As a first step, we claim that 
    \begin{equation}\label{>0}
        {\rm\underline{dist}}(\overline{R(x_t)},\partial  I^{+}_{A}(x))>0 \text{ for any }t>1.
    \end{equation}
    In order to obtain the claim, we only need to prove $\overline{R(x_t)}\cap\partial  I^{+}_{A}(x)=\emptyset$ since it implies \eqref{>0} immediately by the definition of $\underline{dist}$ and the compactness of $\overline{R(x_t)}$ and $\partial  I^{+}_{A}(x)$. Now we prove $\overline{R(x_t)}\cap\partial  I^{+}_{A}(x)=\emptyset$. By Proposition \ref{quasi-closed-prop}, it has $x_{t}<_M y$ for any $y\in \overline{I^{+}(x_{t})}\backslash\{x_{t}\}$ and any $t>0$. Then it follows from Proposition \ref{monotone property}(i) that $\varphi_1  (\overline{I^{+}(x_{t-1})}\backslash\{x_{t-1}\})\subset  I^{+}(x_{t})$ for any $t>1$. Since $I^{+}_{r}(x_{t-1})\subset I^{+}(x_{t-1})$ and $\varphi_1(x_{t-1})=x_{t}$, it has $\varphi_1(\overline{I^{+}_{r}(x_{t-1})})\subset  I^{+}(x_{t})\cup \{x_{t}\}$, and hence, $\overline{R(x_t)}\subset g_{x_{t}}^{x}(I^{+}(x_{t})\cup \{x_{t}\})$ for any $t>1$. So, by (H1) (or rather, Proposition \ref{homogeneous-prop}), it has  $\overline{R(x_t)}\subset I^{+}(x)\cup \{x\}$ for any $t>1$. Since $\partial  I^{+}_{A}(x)\subset\partial  I^{+}(x)\backslash\{x\}$ and $( I^{+}(x)\cup\{x\})\cap (\partial  I^{+}(x)\backslash\{x\})=\emptyset$, it has $\overline{R(x_t)}\cap\partial  I^{+}_{A}(x)=\emptyset$.

    As a second step, we show that for any $t>1$, ${\rm\underline{dist}}(\overline{R(x_t)},\partial  I^{+}_{A}(x))$ is continuous with respect to $t$, or equally, ${\rm\underline{dist}}(g_{x_{t}}^{x}\varphi_1 (I^{+}_{r}(x_{t-1})), \partial  I^{+}_{A}(x))$ is continuous with respect to $t$. Fix $t_0>1$. By \eqref{dis_2}, it has 
    \begin{align}\label{t0t}
    &\lim_{t\rightarrow t_{0}}|{\rm\underline{dist}}(g_{x_{t}}^{x}\varphi_1 (I^{+}_{r}(x_{t-1})), \partial  I^{+}_{A}(x))-{\rm\underline{dist}}(g_{x_{t_{0}}}^{x}\varphi_1 (I^{+}_{r}(x_{t_{0}-1})), \partial  I^{+}_{A}(x))|\notag\\
    \leq &\lim_{t\rightarrow t_{0}}d_{\mathcal{H}}(g_{x_{t}}^{x}\varphi_1 (I^{+}_{r}(x_{t-1})),g_{x_{t_{0}}}^{x}\varphi_1 (I^{+}_{r}(x_{t_{0}-1})))\notag\\
    \leq &\underbrace{\lim_{t\rightarrow t_{0}}d_{\mathcal{H}}(g_{x_{t}}^{x}\varphi_1 (I^{+}_{r}(x_{t-1})),g_{x_{t_{0}}}^{x}\varphi_1 (I^{+}_{r}(x_{t-1})))}_{\text{I}}+\underbrace{\lim_{t\rightarrow t_{0}}d_{\mathcal{H}}(g_{x_{t_{0}}}^{x}\varphi_1 (I^{+}_{r}(x_{t-1})),g_{x_{t_{0}}}^{x}\varphi_1 (I^{+}_{r}(x_{t_{0}-1})))}_{\text{II}}.
    \end{align}
    
    First, we show the part II of \eqref{t0t} is $0$. By (H3) (i.e., the continuity of $I^{+}$) and the compactness of $\overline{I^{+}_{r}(x_{t})}$ for any $t>1$, it follows from \cite[Corollary 5.21]{R-09} that
    \begin{equation}\label{tt0}
        \lim_{t\rightarrow t_{0}}d_{\mathcal{H}}(I^{+}_{r}(x_{t-1}),I^{+}_{r}(x_{t_{0}-1}))=0.
    \end{equation}
    Thus by the continuity of the system and the group action, we obtain 
    \begin{equation}\label{b}
        \lim_{t\rightarrow t_{0}}d_{\mathcal{H}}(g_{x_{t_{0}}}^{x}\varphi_1 (I^{+}_{r}(x_{t-1})),g_{x_{t_{0}}}^{x}\varphi_1 (I^{+}_{r}(x_{t_{0}-1})))=0,
    \end{equation}
    that is, the part II of \eqref{t0t} is $0$.
    
    Next, we show the part I of \eqref{t0t} is $0$. For any $t>1$, since $d(g_{x_{t}}^{x}a,g_{x_{t_{0}}}^{x}\varphi_1 (I^{+}_{r}(x_{t-1})))=\inf_{b\in \varphi_1 (I^{+}_{r}(x_{t-1}))} d(g_{x_{t}}^{x}a,g_{x_{t_{0}}}^{x}b)$ for any $a\in \varphi_1(I^{+}_{r}(x_{t-1}))$, one has 
    $d(g_{x_{t}}^{x}a,g_{x_{t_{0}}}^{x}\varphi_1 (I^{+}_{r}(x_{t-1})))\leq d(g_{x_{t}}^{x}a,g_{x_{t_{0}}}^{x}a)$ for any $a\in \varphi_1(I^{+}_{r}(x_{t-1}))$. Similarly, for $t>1$, it has 
    $d(g_{x_{t_{0}}}^{x}a,g_{x_{t}}^{x}\varphi_1 (I^{+}_{r}(x_{t-1})))\leq d(g_{x_{t_{0}}}^{x}a,g_{x_{t}}^{x}a)$ for any $a\in \varphi_1(I^{+}_{r}(x_{t-1}))$. For any $t>1$, choose $a_t\in \varphi_1(\overline{I^{+}_{r}(x_{t-1})})$ such that $\sup_{a\in\varphi_1 (I^{+}_{r}(x_{t-1}))}d(g_{x_{t_{0}}}^{x}a,g_{x_{t}}^{x}a)\leq d(g_{x_{t_{0}}}^{x}a_t,g_{x_{t}}^{x}a_t)$. Then for $t>1$, it has
    \begin{equation}\label{at}
    d_{\mathcal{H}}(g_{x_{t}}^{x}\varphi_1 (I^{+}_{r}(x_{t-1})),g_{x_{t_{0}}}^{x}\varphi_1 (I^{+}_{r}(x_{t-1})))\leq d(g_{x_{t}}^{x}a_t,g_{x_{t_{0}}}^{x}a_t).
    \end{equation} 
    By the compactness of $\overline{I^{+}_{r}(x_{t_{0-1}})}$, there exists $\kappa>0$ such that $B_{\kappa}(\overline{I^{+}_{r}(x_{t_0-1})})$ is compact, where $B_{\kappa}(\overline{I^{+}_{r}(x_{t_0-1})})\triangleq\{v: d(v,\overline{I^{+}_{r}(x_{t_0-1})})\leq \kappa\}$. 
    Noting $\lim_{t\rightarrow t_{0}}d_G(g_{x_{t}}^{x},g_{x_{t_{0}}}^{x})=0$ by Proposition \ref{distance-limit}(iii), it has $\lim_{t\rightarrow t_{0}}d(g_{x_{t}}^{x}a,g_{x_{t_{0}}}^{x}a)=0$ uniformly for any $a\in B_{\kappa}(\overline{I^{+}_{r}(x_{t_0-1})})$. Together with \eqref{tt0} and fact that $a_t\in \varphi_1(\overline{I^{+}_{r}(x_{t-1})})$, we obtain $\lim_{t\rightarrow t_{0}}d(g_{x_{t}}^{x}a_t,g_{x_{t_{0}}}^{x}a_t)=0$, and hence, with \eqref{at}, we have
    \begin{equation}\label{a}
        \lim_{t\rightarrow t_{0}}d_{\mathcal{H}}(g_{x_{t}}^{x}\varphi_1 (I^{+}_{r}(x_{t-1})),g_{x_{t_{0}}}^{x}\varphi_1 (I^{+}_{r}(x_{t-1})))=0.
    \end{equation}
    Thus, with \eqref{b} and \eqref{a}, we calculate \eqref{t0t} and obtain
    \begin{equation*}
    \lim_{t\rightarrow t_{0}}|{\rm\underline{dist}}(g_{x_{t}}^{x}\varphi_1 (I^{+}_{r}(x_{t-1})), \partial  I^{+}_{A}(x))-{\rm\underline{dist}}(g_{x_{t_{0}}}^{x}\varphi_1 (I^{+}_{r}(x_{t_{0}-1})), \partial  I^{+}_{A}(x))|=0.
    \end{equation*}
    This implies the continuity of ${\rm\underline{dist}}(g_{x_{t}}^{x}\varphi_1 (I^{+}_{r}(x_{t-1})), \partial  I^{+}_{A}(x))$ (or ${\rm\underline{dist}}(\overline{R(x_t)},\partial  I^{+}_{A}(x))$) with respect to $t$.

    Finally, we prove \eqref{>delta}. Note the continuity of ${\rm\underline{dist}}(\overline{R(x_t)},\partial  I^{+}_{A}(x))$ with respect to $t$ is same as the continuity of ${\rm\underline{dist}}(\overline{R(x_t)},\partial  I^{+}_{A}(x))$ with respect to $x_t$. Together with the compactness of $\overline{O^{+}(x)}$, it can enhance the non-negativity in \eqref{>0} to be positive, which is \eqref{>delta}. This completes the proof.
\end{proof}

\begin{prop}\label{W_{x}}
    There are some $\eta>0$ and $t_0>1$ such that ${\rm \underline{dist}}(\varphi_{t_{0}}(W_{x}),\partial  I^{+}(x))>\eta$ and $\varphi_{t_{0}}(W_{x})\subset W_{x}$.
\end{prop}

\begin{proof}
    Recalling $\omega(z)\subset\text{Int}A$, it follows that $\omega(z)\cap (\partial  I^{+}(x)\backslash\text{Int}A)=\emptyset$, which implies that ${\rm \underline{dist}}(\omega(z),\partial  I^{+}(x)\backslash\text{Int}A)>0$. Choose an $\eta>0$ small that $$\eta<\min\left(\frac{\delta}{2}, {\rm \underline{dist}}(\omega(z),\partial  I^{+}(x)\backslash\text{Int}A)\right),$$
    where $\delta$ is given in Proposition \ref{R(x_t)}. Since $\lim_{x_t\rightarrow x}d_{G}(g^{x}_{x_t},e)=0$ by the Proposition \ref{distance-limit}(ii), it has $\lim_{x_t\rightarrow x}d(a,g^{x}_{x_t}a)=0$ for any $a\in \omega(z)$. Thus, by the compactness of $\omega(z)$ and the fact $x\in \mathcal{R}(\varphi)$, there exists $t_0>1$ such that $d(a,g^{x}_{x_{t_0}}a)<\eta$ for any $a\in \omega(z)$. Noting $W_{x}^{t_0}=g^{x}_{x_{t_0}}\varphi_{t_{0}}(W_{x})$, there exists $a_0\in \varphi_{t_{0}}(W_{x})$ such that $d_{\mathcal{H}}(\varphi_{t_{0}}(W_{x}),W_{x}^{t_0})\le d(a_0,g^{x}_{x_{t_0}}a_0)$. Recalling $\varphi_{t_{0}}(W_{x})\subset \omega(z)$ which means $a_0\in \omega(z)$, it has 
    \begin{equation}\label{le}
        d_{\mathcal{H}}(\varphi_{t_{0}}(W_{x}),W_{x}^{t_0})\le \eta.
    \end{equation}
    Consequently, together with inequality (\ref{dis_2}) and Proposition \ref{R(x_t)}, we obtain that
		\begin{align}
			{\rm \underline{dist}}\left(\varphi_{t_{0}}(W_{x}), \partial  I^{+}_{A}(x)\right)
			&\geq {\rm \underline{dist}}\left(W_{x}^{t_{0}}, \partial  I^{+}_{A}(x)\right)-d_{\mathcal{H}}(\varphi_{t_{0}}(W_{x}),W_{x}^{t_{0}})>\delta-\eta>\eta. \label{E:W-eta-L}
		\end{align}
		Note also that $\eta<{\rm \underline{dist}}(\omega(z),\partial  I^{+}(x)\backslash{\rm Int}A))$ and $\varphi_{t_{0}}(W_{x})\subset\omega(z)$, we further obtain $${\rm \underline{dist}}(\varphi_{t_{0}}(W_{x}),\partial  I^{+}(x)\backslash{\rm Int}A))>\eta.$$ Together with \eqref{E:W-eta-L}, this entails that
		\begin{equation}\label{E:W-boundary}
			{\rm \underline{dist}}(\varphi_{t_{0}}(W_{x}),\partial  I^{+}(x))>\eta.
		\end{equation}

        As for $\varphi_{t_{0}}(W_{x})\subset W_{x}$, by the invariance of $\omega(z)$ and the definition of $W_{x}$ (i.e., $W_x= \overline{I^{+}(x)}\cap\omega(z)$), it suffices to prove that $\varphi_{t_{0}}(W_{x})\subset  \overline{I^{+}(x)}$. We prove it by contradiction. Suppose that $\varphi_{t_{0}}(W_{x})\not\subset  \overline{I^{+}(x)}$, that is, there is $y\in \varphi_{t_{0}}(W_{x})\cap (M\backslash  \overline{I^{+}(x)})$.
		It follows from $y\in \varphi_{t_{0}}(W_{x})$ and \eqref{E:W-boundary} that ${\rm \underline{dist}}(\{y\},\partial  I^{+}(x))>\eta$. Moreover, it has ${\rm \underline{dist}}(\{y\}, \overline{I^{+}(x)})>\eta$, since $d(y,\partial  I^{+}(x))=d(y,\overline{I^{+}(x)})$ deduced form $y\notin \overline{I^{+}(x)}$. Together with Proposition \ref{quasi-closed-prop} and the strong positivity of the system, Proposition \ref{homogeneous-prop} entails $W_{x}^{t_{0}}\subset \overline{I^{+}(x)}$, and hence, one has 
		\begin{equation}\label{y_Wt0x>eta}
			{\rm \underline{dist}}(\{y\}, W_{x}^{t_{0}})>\eta.
		\end{equation}
		However, noticing that $y\in \varphi_{t_{0}}(W_{x})$, it follows from \eqref{le} that $${\rm \underline{dist}}(\{y\}, W_{x}^{t_{0}})\leq d_{\mathcal{H}}(\varphi_{t_{0}}(W_{x}),W_{x}^{t_{0}})\leq\eta,$$
		which clearly contradicts \eqref{y_Wt0x>eta}.
		Thus, we have proved $\varphi_{t_{0}}(W_{x})\subset  \overline{I^{+}(x)}$, and hence, $\varphi_{t_{0}}(W_{x})\subset W_{x}$, which completes the proof.
\end{proof}

Now, we are ready to prove the Lemma \ref{sep_n}.

\begin{proof}[Proof of Lemma \ref{sep_n}]	
    (i) $x\in \partial I^{+}(x)\cup \partial I^{-}(x)$. By the Proposition \ref{chronological boundary}, it has $x\in \partial I^{+}(x)$. Thus, to prove $\mathcal{R}(\varphi)\cap \partial I^{+}(x)=\{x\}$, it suffices to prove $\mathcal{R}(\varphi)\cap (\partial I^{+}(x)\backslash\{x\})=\emptyset$. We prove it by contradiction. Suppose that there is $z\in \mathcal{R}(\varphi)$ with $x\neq z$ and $z\in \partial  I^{+}(x)$. By Proposition \ref{W_{x}}, there are some $\eta>0$ and $t_0>1$ such that $${\rm \underline{dist}}(\varphi_{t_{0}}(W_{x}),\partial  I^{+}(x))>\eta.$$
    Note $\lim_{t\rightarrow0}d_{\mathcal{H}}(\varphi_{t_{0}}(W_{x}), \varphi_t(\varphi_{t_{0}}(W_{x})))=0$ (it follows from $\lim_{t\rightarrow0} d(u, \varphi_t (\varphi_{t_{0}}(W_{x})))=0$ and $\lim_{t\rightarrow0} d(\varphi_{t_{0}}(W_{x}), \varphi_t(u))=0$ for any $u\in \varphi_{t_{0}}(W_{x})$). There exists $\varepsilon_{0}>0$ such that for any $|t|<\varepsilon_{0}$,
	\begin{equation*}
		d_{\mathcal{H}}\left(\varphi_{t_{0}}(W_{x}), \varphi_{t_{0}+t} (W_{x})\right)<\frac{1}{2}\eta.
	\end{equation*}
    Furthermore, by (\ref{dis_2}), for any $|t|<\varepsilon_{0}$, one has
	\begin{equation}\label{dis_betw_Wx_pCx}
	    {\rm \underline{dist}}(\varphi_{t_{0}+t} (W_{x}), \partial  I^{+}(x))
        \geq {\rm \underline{dist}}(\varphi_{t_{0}}(W_{x}), \partial  I^{+}(x))-d_{\mathcal{H}}(\varphi_{t_{0}}(W_{x}),\varphi_{t_{0}+t} (W_{x}))
        \geq \frac{1}{2}\eta
	\end{equation}
	Denote the time-set by $\mathcal{T}(t_0,\varepsilon_0)\triangleq\{nt_0+t:n\in\mathbb{N},|t|<\varepsilon_0\}$. 
	Together with $\varphi_{t_{0}}(W_{x})\subset W_{x}$ in Proposition \ref{W_{x}}, \eqref{dis_betw_Wx_pCx} entails
	\begin{equation}\label{ds}
		{\rm \underline{dist}}(\varphi_{s} (W_{x}), \partial  I^{+}(x))>\frac{1}{2}\eta,\quad \text{ for any } s\in\mathcal{T}(t_0,\varepsilon_0).
	\end{equation}
	Recall that $z\in W_x$ and $z\in \partial  I^{+}(x)$. Then, by \eqref{ds}, we have that $$d(\varphi_{s}(z),z)>\frac{1}{2}\eta,\quad \text{ for any } s\in \mathcal{T}(t_0,\varepsilon_0).$$
	Thus, $$\{t>0: d(\varphi_{t}(z),z)<\frac{1}{2}\eta\}\cap \mathcal{T}(t_{0},\varepsilon_{0})=\emptyset.$$ This contradicts to the properties of the IP-set, as further explained in the remark below. Therefore, it has $\mathcal{R}(\varphi)\cap (\partial I^{+}(x)\backslash\{x\})=\emptyset$, and hence, $\mathcal{R}(\varphi)\cap \partial I^{+}(x)=\{x\}$. With a similar argument, we obtain $\mathcal{R}(\varphi)\cap \partial I^{-}(x)=\{x\}$. Thus, we have proved $\mathcal{R}(\varphi)\cap (\partial  I^{+}(x)\cup \partial  I^{-}(x))=\{x\}$ if $x\in \partial I^{+}(x)\cup \partial I^{-}(x)$.

    (ii) $x\in  I^{+}(x)\cup I^{-}(x)$. We first prove $\mathcal{R}(\varphi)\cap \partial  I^{+}(x)=\emptyset$ by contradiction. Suppose that there is a point $z\in \mathcal{R}(\varphi)\cap\partial  I^{+}(x)$. Clearly, $z\neq x$ (since $x\in I^{+}(x)$ by the Proposition \ref{chronological boundary} (i) and the assumption $x\in I^{+}(x)\cup I^{-}(x)$). Thus, following the proof in (i), one can obtain a contradiction to the properties of the IP-set. Thus, we have proved $\mathcal{R}(\varphi)\cap \partial  I^{+}(x)=\emptyset$ if $x\in I^{+}(x)\cup I^{-}(x)$. Similarly, one can obtain $\mathcal{R}(\varphi)\cap \partial  I^{-}(x)=\emptyset$. Thus, we have proved $\mathcal{R}(\varphi)\cap (\partial  I^{+}(x)\cup \partial  I^{-}(x))=\emptyset$ if $x\in I^{+}(x)\cup I^{-}(x)$.	
\end{proof}

\begin{rmk}
    We further elaborate on the properties of IP-sets in this proof (see more details in \cite[Appendix A]{SWZ-25} and \cite{F81}). We say a set $S\subset\mathbb{R}^+$ is an IP-set if there is a sequence $\{p_i\}_{i\geq1}\subset\mathbb{R}^+$ such that $$S =\{ p_{i_1} + \cdots + p_{i_k} : i_1 < \cdots < i_k,\ k \ge 1\}.$$ By the construction of IP-set $S$, it can be proved (see the proof in \cite[Appendix A]{SWZ-25}) that for any $\tau>0$ and $\varepsilon>0$, $$S\cap \mathcal{T}(\tau,\varepsilon)\neq\emptyset, \text{ where } \mathcal{T}(\tau,\varepsilon)=\{n\tau+t:n\in\mathbb{N},|t|<\varepsilon\}.$$ For any recurrent point $z$ and any recurrent-time set $N(z,\theta)$ of $z$ parameterized by $\theta$ $$N(z,\theta)\triangleq\{t>0: d(\varphi_{t}(z),z)<\theta\},$$ we can construct an IP-set in it (see more details about construction in \cite[Appendix A]{SWZ-25} or \cite[Proposition 8.10]{F81}). Thus, $N(z,\theta)\cap \mathcal{T}(\tau,\varepsilon)\neq\emptyset$.
\end{rmk}

\subsection{Limit-set dichotomy for recurrent points}\label{s3.3}
In this subsection, we present a technical lemma, called as ``\textit{limit-set dichotomy for recurrent points}", which turns out to be crucial to obtain the intersection principle.

\begin{lem}\label{limit_set_dichotomy}
	Let $x,y\in \mathcal{R}(\varphi)$. Then either \begin{enumerate}[{\rm (i)}]
	\item $\omega(x)$, $\omega(y)$ are strongly ordered, or
    \item $\omega(x)$, $\omega(y)$ are unordered.
    \end{enumerate}
\end{lem}
Here, $\omega(x)$, $\omega(y)$ are called \textit {strongly ordered} if for any $p\in \omega(x)$ and any $q\in \omega(y)$ with $p\neq q$, $p$ and $q$ are strongly ordered; $\omega(x)$, $\omega(y)$ are called \textit{unordered}, if for any $p\in \omega(x)$ and any $q\in \omega(y)$ with $p\neq q$, $p$ and $q$ are unordered.

Before proving it, we need some technical tools.
	\begin{prop}\label{sep_n_1}
		Let $x\in \mathcal{R}(\varphi)$ with $x\in \partial I^{+}(x)\cup\partial I^{-}(x)$. Then, for any $y\in \mathcal{R}(\varphi)$, 
		\begin{equation*}
			\omega(y)\cap(\partial I^{+}(x)\backslash\{x\})=\emptyset \text{~and~} \omega(y)\cap(\partial I^{-}(x)\backslash\{x\})=\emptyset.
		\end{equation*}
	\end{prop}

	\begin{proof}
		Suppose $x\in \mathcal{R}(\varphi)$ with $x\in \partial I^{+}(x)\cup\partial I^{-}(x)$. Then by the Proposition \ref{chronological boundary}, it has $x\in \partial I^{+}(x)\cap\partial I^{-}(x)$. Here, we only prove the case 
        \begin{equation}\label{omega_R}
			\omega(y)\cap(\partial  I^{+}(x)\backslash\{x\})=\emptyset,
		\end{equation}
        and the another case can be deduced by analogy. In order to obtain \eqref{omega_R}, we consider the following four cases, respectively.

        Case (i): $x\in\omega(y)$. The Proposition \ref{ou} implies \eqref{omega_R} immediately. (If not, there exists $a\in \omega(y)\cap\partial  I^{+}(x)$ with $a\neq x$, which means $a<_M x$ but not $a\ll_M x$ by Proposition \ref{quasi-closed-prop}, a contradiction to Proposition \ref{ou}.)

        Case (ii): $x\notin\omega(y)$ and $\omega(y)\cap  I^{+}(x)\neq\emptyset$.
		We show \eqref{omega_R} by contradiction.
		Suppose that there exists $y_{0} \in\omega(y)\cap\partial  I^{+}(x)$. Clear, $y_{0}\neq x$. Recalling $\omega(y)\cap  I^{+}(x)\neq\emptyset$, there is $y_{1} \in\omega(y)$ such that $y_{1}\in  I^{+}(x)$.
        Then, one can take a small $\tau>0$ such that $\varphi_{-\tau}(y_{1})\in  I^{+}(\varphi_{-\tau}(x))$ by (H3) (i.e., continuity of $I^{+}(x)$) and $\varphi_{-\tau}(y_{0})\notin  \overline{I^{+}(\varphi_{-\tau}(x))}$ by (H2) (or rather, Proposition \ref{quasi-closed-prop}) and the strong positivity of the system.
		Choose $z_i\in O^{+}(y)$ close to $\varphi_{-\tau}(y_i)$ for $i=0,1$,
		such that $z_1\in  I^{+}(\varphi_{-\tau}(x))$ and $z_0\notin  \overline{I^{+}(\varphi_{-\tau}(x))}$. Then, the connectedness of $O^{+}(y)$ implies that there exists a point $z\in O^{+}(y)$ such that $z\in \partial  I^{+}(\varphi_{-\tau}(x))$. Furthermore, $z\neq\varphi_{-\tau}(x)$ as $\varphi_{-\tau}(x)\notin O^{+}(y)$. Noticing that $z\in O^{+}(y)\subset \mathcal{R}(\varphi)$, we obtain a contradiction to Lemma \ref{sep_n}(i). Therefore, we obtain \eqref{omega_R}.

        Case (iii): $x\notin\omega(y)$ and $\omega(y)\not\subset  \overline{I^{+}(x)}$.
		We also show \eqref{omega_R} by contradiction.
		Suppose that there exists $y_{0} \in\omega(y)\cap\partial  I^{+}(x)$. Clear, $y_{0}\neq x$. Recalling $\omega(y)\not\subset  \overline{I^{+}(x)}$, there is $y_{2} \in\omega(y)$ such that $y_{2}\notin  \overline{I^{+}(x)}$.
        Choose $\tau'>0$ small so that $\varphi_{\tau'}(y_{2})\notin  \overline{I^{+}(\varphi_{\tau'}(x))}$ by (H3) (i.e., the continuity of $I^{+}$) and $\varphi_{\tau'}(y_{0})\in  I^{+}(\varphi_{\tau'}(x))$ by (H2) (or rather, Proposition \ref{quasi-closed-prop}) and the strong positivity of the system. Choose $z'_i\in O^{+}(y)$ close to $\varphi_{\tau'}(y_i)$ for $i=0,2$,
		such that $z'_0\in  I^{+}(\varphi_{\tau'}(x))$ and $z'_2\notin  \overline{I^{+}(\varphi_{\tau'}(x))}$. Then by the connectedness of $O^{+}(y)$, there is a point $z'\in O^{+}(y)$ such that $z'\in \partial  I^{+}(\varphi_{\tau'}(x))$. Moreover, since $\varphi_{\tau'}(x)\notin O^{+}(y)$, it has $z'\neq\varphi_{\tau'}(x)$. Noticing that $z'\in O^{+}(y)\subset \mathcal{R}(\varphi)$, one can obtain a contradiction to Lemma \ref{sep_n}(i). Therefore, we obtain \eqref{omega_R}.
        
        Case (iv): $x\notin\omega(y)$ and $\omega(y)\subset  \partial{I^{+}(x)}$. We show this case actually cannot happen. In fact, since $y\in \mathcal{R}(\varphi)$, it has $y\in \omega(y)\subset  \partial{I^{+}(x)}$. Noting $x\notin\omega(y)$ (which entails $x\neq y$), so $y\in \partial{I^{+}}(x)\backslash\{x\}$, which contradicts to Lemma \ref{sep_n}(i).

        This completes the proof.
	\end{proof}

    \begin{prop}\label{sep_n_2}
		Let $x\in \mathcal{R}(\varphi)$ with $x\in I^{+}(x)\cup I^{-}(x)$. Then for any $y\in \mathcal{R}(\varphi)$,
        \begin{equation*}
			\omega(y)\cap\partial I^{+}(x)=\emptyset \text{~and~} \omega(y)\cap\partial I^{-}(x)=\emptyset.
		\end{equation*}
	\end{prop}

    \begin{proof}
		Suppose $x\in \mathcal{R}(\varphi)$ with $x\in I^{+}(x)\cup I^{-}(x)$. Then it follows from the Proposition \ref{chronological boundary} that $x\in I^{+}(x)\cap I^{-}(x)$. Here, we only prove the case $\omega(y)\cap\partial  I^{+}(x)=\emptyset$, and the another case can be deduced by analogy. Similar to the Proposition \ref{sep_n_1}, we consider the following four cases, respectively.

        Case (i): $x\in\omega(y)$. We will show $\omega(y)\cap\partial  I^{+}(x)=\emptyset$ by contradiction. Suppose there is $a\in \omega(y)\cap\partial  I^{+}(x)$. Clearly, $a\neq x$ as $x\in I^{+}(x)$. Then by Proposition \ref{quasi-closed-prop}, it has $a<_M x$ but not $a\ll_M x$, a contradiction to Proposition \ref{ou}. Therefore, we obtain $\omega(y)\cap\partial  I^{+}(x)=\emptyset$.

        Case (ii): $x\notin\omega(y)$ and $\omega(y)\cap  I^{+}(x)\neq\emptyset$.
		We show $\omega(y)\cap\partial  I^{+}(x)=\emptyset$ by contradiction.
		Suppose there exists $y_{0} \in\omega(y)\cap\partial  I^{+}(x)$. Clearly, it has $y_{0}\neq x$ since $x\in I^{+}(x)$. Since $\omega(y)\cap  I^{+}(x)\neq\emptyset$ and $x\notin\omega(y)$, there is $y_{1} \in\omega(y)\cap I^{+}(x)$. Then, one can take a small $\tau>0$ such that $\varphi_{-\tau}(y_{1})\in  I^{+}(\varphi_{-\tau}(x))$ by (H3) (i.e., continuity of $I^{+}(x)$) and $\varphi_{-\tau}(y_{0})\notin  \overline{I^{+}(\varphi_{-\tau}(x))}$ by (H2) (or rather, Proposition \ref{quasi-closed-prop}) and the strong positivity of the system.
		Choose $z_i\in O^{+}(y)$ close to $\varphi_{-\tau}(y_i)$ for $i=0,1$,
		such that $z_1\in  I^{+}(\varphi_{-\tau}(x))$ and $z_0\notin  \overline{I^{+}(\varphi_{-\tau}(x))}$. Then, the connectedness of $O^{+}(y)$ implies that there exists a point $z\in O^{+}(y)$ such that $z\in \partial  I^{+}(\varphi_{-\tau}(x))$. Furthermore, $z\neq\varphi_{-\tau}(x)$ as $\varphi_{-\tau}(x)\notin O^{+}(y)$. Noticing that $z\in O^{+}(y)\subset \mathcal{R}(\varphi)$, we obtain a contradiction to Lemma \ref{sep_n}(ii). Therefore, we obtain
        $\omega(y)\cap\partial  I^{+}(x)=\emptyset$.
        
        %Similar to the proof of Proposition \ref{sep_n_1}, one can take a small $\tau>0$ and a point $z\in O^{+}(y)$ such that $z\in \partial  I^{+}(\varphi_{-\tau}(x))$ and $z\neq\varphi_{-\tau}(x)$. Noticing that $z\in O^{+}(y)\subset \mathcal{R}(\varphi)$, we obtain a contradiction to Lemma \ref{sep_n}(ii). Therefore, we obtain $\omega(y)\cap\partial  I^{+}(x)=\emptyset$.

        Case (iii): $x\notin\omega(y)$ and $\omega(y)\not\subset  \overline{I^{+}(x)}$.
		We also show $\omega(y)\cap\partial  I^{+}(x)=\emptyset$ by contradiction.
		Suppose there exists $y_{0} \in\omega(y)\cap\partial  I^{+}(x)$. Clearly, $y_{0}\neq x$.
        Recalling $\omega(y)\not\subset  \overline{I^{+}(x)}$, there is $y_{2} \in\omega(y)$ such that $y_{2}\notin  \overline{I^{+}(x)}$. Choose $\tau'>0$ small so that $\varphi_{\tau'}(y_{2})\notin  \overline{I^{+}(\varphi_{\tau'}(x))}$ by (H3) (i.e., the continuity of $I^{+}$) and $\varphi_{\tau'}(y_{0})\in  I^{+}(\varphi_{\tau'}(x))$ by (H2) (or rather, Proposition \ref{quasi-closed-prop}) and the strong positivity of the system. Choose $z'_i\in O^{+}(y)$ close to $\varphi_{\tau'}(y_i)$ for $i=0,2$,
		such that $z'_0\in  I^{+}(\varphi_{\tau'}(x))$ and $z'_2\notin  \overline{I^{+}(\varphi_{\tau'}(x))}$. Then by the connectedness of $O^{+}(y)$, there is a point $z'\in O^{+}(y)$ such that $z'\in \partial  I^{+}(\varphi_{\tau'}(x))$. Moreover, since $\varphi_{\tau'}(x)\notin O^{+}(y)$, it has $z'\neq\varphi_{\tau'}(x)$. Noticing that $z'\in O^{+}(y)\subset \mathcal{R}(\varphi)$, one can obtain a contradiction to Lemma \ref{sep_n}(ii).
        Therefore, we obtain $\omega(y)\cap\partial  I^{+}(x)=\emptyset$.
        
        Case (iv): $x\notin\omega(y)$ and $\omega(y)\subset  \partial{I^{+}(x)}$. We will show this case cannot happen. Actually, since $y\in \omega(y)$ (as $y\in \mathcal{R}(\varphi)$), it has $y\in  \mathcal{R}(\varphi)\cap\partial{I^{+}(x)}$, which contradicts to Lemma \ref{sep_n}(ii).
        
        This completes the proof.
	\end{proof}

Combining Propositions \ref{sep_n_1} and \ref{sep_n_2}, we obtain the following corollary:

	\begin{cor}\label{absorb}
		Let $x,y\in \mathcal{R}(\varphi)$ with $x\notin\omega(y)$. Whether $x\in \partial I^{+}(x)\cup \partial I^{-}(x)$ or $x\in I^{+}(x)\cup I^{-}(x)$, the following results hold: 
        \begin{enumerate}[{\rm (i)}]
            \item If $\omega(y)\cap I^{+}(x)\neq\emptyset$, then $\omega(y)\subset I^{+}(x)$;
            \item If $\omega(y)\cap I^{-}(x)\neq\emptyset$, then $\omega(y)\subset I^{-}(x)$.
        \end{enumerate}
	\end{cor}

	\begin{proof}
		We show case (i) by contradiction. Suppose $\omega(y)\not\subset I^{+}(x)$ with $\omega(y)\cap I^{+}(x)\neq\emptyset$. Then by the connectedness of $\omega(y)$, it has $\omega(y)\cap\partial I^{+}(x)\neq\emptyset$. Moreover, since $x\notin\omega(y)$, it has $\omega(y)\cap(\partial I^{+}(x)\backslash\{x\})\neq\emptyset$, which is a contradiction to Proposition \ref{sep_n_1} (if $x\in \partial I^{+}(x)\cup \partial I^{-}(x)$) or Proposition \ref{sep_n_2} (if $x\in I^{+}(x)\cup I^{-}(x)$). Thus, we have obtained the case (i). Case (ii) can be deduced by analogy.
	\end{proof}

    Now, we are ready to prove the limit-set dichotomy for recurrent points.
	
	\begin{proof}[Proof of Lemma \ref{limit_set_dichotomy}]
    We first consider the case that one of $x,y$ is an equilibrium. Without loss of generality, we assume $x$ is an equilibrium, then $x=\omega(x)$. If $x\in \omega(y)$, Proposition \ref{ou} implies that any points in $\omega(y)$ are ordered with $\omega(x)(=x)$. Therefore, we obtain the lemma.
    If $x\notin \omega(y)$, it follows from Propositions \ref{sep_n_1} and \ref{sep_n_2} that $\omega(y)\cap(\partial I^{+}(x)\cup\partial I^{-}(x))=\emptyset$. Then Corollary \ref{absorb} entails that one of the following alternatives must occur: {\rm (1)} $\omega(y)\subset I^{+}(x)$; {\rm (2)} $\omega(y)\subset I^{-}(x)$; {\rm (3)} $\omega(y)\subset M\backslash \overline{I^{+}(x)\cup I^{-}(x)}$.
    So, the lemma follows. In the remainder, we assume both $x,y\in \mathcal{R}(\varphi)$ are not equilibrium.

	(i) $\omega(x)\cap\omega(y)\neq\emptyset$. Suppose $a\in \omega(x)\cap\omega(y)$. Proposition \ref{ou} shows that $\omega(x)$ is either strongly ordered or unordered, as well as $\omega(y)$. Thus, if $\omega(x)\subset\omega(y)$ or $\omega(y)\subset\omega(x)$, then the lemma can be obtained directly. So in the remainder of this case, we assume $x\notin \omega(y)$ and $y\notin \omega(x)$. We consider the following two situations: (ia) $\omega(x)$ is strongly ordered; (ib) $\omega(x)$ is unordered.

    (ia) $\omega(x)$ is strongly ordered. We claim that $\omega(y)$ is strongly ordered. If the claim holds, we prove the conclusion of the lemma. Since $a\in \omega(y)$, it follows from the proof of Proposition \ref{ou} that $a\in I^{+}(y')\cup I^{-}(y')$ for any $y'\in \omega(y)\backslash\{a\}$. 
    This, with Corollary \ref{absorb}, implies that $\omega(x)\subset I^{+}(y')\cup I^{-}(y')$ for any point $y'\in \omega(y)\backslash\omega(x)$. 
    As for $y'\in \omega(y)\cap\omega(x)$, it also follows from the proof of Proposition \ref{ou} that $\omega(x)\backslash \{y'\}\subset I^{+}(y')\cup I^{-}(y')$. 
    Thus, we obtain $\omega(x)$, $\omega(y)$ are strongly ordered. Now, we prove the claim. 
    In fact, if $\omega(y)$ is unordered, then there is $y_1\in \omega(y)$ such that $a\notin J^{+}(y_1)\cup J^{-}(y_1)$. So, $a\notin \overline{I^{+}(y_1)}\cup \overline{I^{-}(y_1)}$.
    Choose a point $x_1\in \omega(x)$ which is close enough to $a$ that $x_1\notin\overline{I^{+}(y_1)}\cup \overline{I^{-}(y_1)}$. 
    Together with the proof of Proposition \ref{ou}, the fact that $x_1,a\in \omega(x)$ implies that $x_1\in I^{+}(a)\cup I^{-}(a)$. 
    Thus, there is a point $x_{\tau_1}$ in $O(x)$ close enough to $x_1$ such that $x_{\tau_1}\in I^{+}(a)\cup I^{-}(a)$ and $x_{\tau_1}\notin\overline{I^{+}(y_1)}\cup \overline{I^{-}(y_1)}$. Note the fact that $x_{\tau_1}\in I^{+}(a)$ ($x_{\tau_1}\in I^{-}(a)$, respectively) implies $a\in I^{-}(x_{\tau_1})$ ($a\in I^{+}(x_{\tau_1})$, respectively) which implies $\omega(y)\cap (I^{+}(x_{\tau_1})\cup I^{-}(x_{\tau_1})) \neq\emptyset$ as $a\in \omega(y)$. Since $x_{\tau_1}\notin\omega(y)$ (as $x_{\tau_1}\in O(x)$ and $x\notin\omega(y)$), it has $\omega(y)\subset I^{+}(x_{\tau_1})\cup I^{-}(x_{\tau_1})$ by Corollary \ref{absorb}. 
    Then as $y_1\in \omega(y)$, we have $y_1\in I^{+}(x_{\tau_1})\cup I^{-}(x_{\tau_1})$, and hence, $x_{\tau_1}\in I^{+}(y_1)\cup I^{-}(y_1)$, which is a contradiction to $x_{\tau_1}\notin\overline{I^{+}(y_1)}\cup \overline{I^{-}(y_1)}$. So, we obtain the claim.

    (ib) $\omega(x)$ is unordered. Then $\omega(y)$ is unordered. (We prove it by contradiction. If $\omega(y)$ is strongly ordered, then there is $y_2\in \omega(y)$ such that $a\in I^{+}(y_2)\cup I^{-}(y_2)$. Recall that $a\in \omega(x)\cap\omega(y)$, one can choose a point $x_2\in \omega(x)$ close enough to $a$ that $x_2\in I^{+}(y_2)\cup I^{-}(y_2)$. Since $\omega(x)$ is unordered, it also has $x_2\notin\overline{I^{+}(a)}\cup \overline{I^{-}(a)}$. Thus, there is a point $x_{\tau_2}$ in $O(x)$ which is close enough to $x_2$ such that $x_{\tau_2}\in I^{+}(y_2)\cup I^{-}(y_2)$ and $x_{\tau_2}\notin\overline{I^{+}(a)}\cup \overline{I^{-}(a)}$. Then by a similar argument in situation (ia) as above, one can obtain a contradiction.) Since $a\in \omega(x)\cap\omega(y)$, it has $a\notin \overline{I^{+}(y')}\cup \overline{I^{-}(y')}$ for any $y'\in \omega(y)\backslash\{a\}$. This, together with Corollary \ref{absorb}, implies that $\omega(x) \cap (I^{+}(y')\cup I^{-}(y'))=\emptyset$ for any point $y'\in \omega(y)\backslash \omega(x)$. As for $y'\in \omega(y)\cap \omega(x)$, the fact that $\omega(x)$ is unordered implies that $(\omega(x)\backslash \{y'\}) \cap (I^{+}(y')\cup I^{-}(y'))=\emptyset$. So, we obtain that $(\omega(x)\backslash \{y'\}) \cap (I^{+}(y')\cup I^{-}(y'))=\emptyset$ for all $y'\in \omega(y)$. Thus, $\omega(x)$, $\omega(y)$ are unordered. (In fact, if there exist $p\in \omega(x)$ and $q\in \omega(y)$ such that $p<_M q$, then the strong positivity of the system entails that $\varphi_{t}(p)\in I^{-}(\varphi_{t}(q))$ for $t>0$, a contradiction.)

    (ii) $\omega(x)\cap\omega(y)=\emptyset$. We assert that \emph{there is no points $x_{0}\in\omega(x)$ and $y_{0}\in\omega(y)$ such that $y_{0}\in \partial I^{+}(x_{0})\cup \partial I^{-}(x_{0})$}. 
	Before proving this assertion, we show how it implies the conclusion of this lemma.
	In fact, choose any $x_*\in\omega(x)$, $y_*\in\omega(y)$. Since $\omega(x)\cap\omega(y)=\emptyset$, one has $x_*\neq y_*$. We consider the following two case:
	(iia) $y_*\in I^{+}(x_*)\cup I^{-}(x_*)$; (iib) $y_1\notin I^{+}(x_*)\cup I^{-}(x_*)$.
		
	(iia) $y_*\in I^{+}(x_*)\cup I^{-}(x_*)$. Then we have
	\begin{equation}\label{x1_approx_omegay}
		\omega(y)\subset I^{+}(x_*)\cup I^{-}(x_*).
	\end{equation}
	(Otherwise, one can find $y_3\in\omega(y)$ such that $y_3\notin I^{+}(x_*)\cup I^{-}(x_*)$. Noting $y_*\in I^{+}(x_*)\cup I^{-}(x_*)$, it follows from the connectedness of $\omega(y)$ that there is $y_4\in\omega(y)$ such that $y_4\in\partial I^{+}(x_*)\cup \partial I^{-}(x_*)$, contradicting the assertion.)
	Consequently, \eqref{x1_approx_omegay} implies that $\omega(y)\subset I^{+}(x')\cup I^{-}(x')$ for any $x'\in \omega(x)$.
	(Otherwise, there exist $x_3\in\omega(x)$ and $y_4\in\omega(y)$ such that $y_4\notin I^{+}(x_3)\cup I^{-}(x_3)$, equivalently $x_3\notin I^{+}(y_4)\cup I^{-}(y_4)$. But, note $x_*\in I^{+}(y_4)\cup I^{-}(y_4)$ by \eqref{x1_approx_omegay}, the connectedness of $\omega(x)$ implies that there is a $x_4\in\omega(x)$ such that $x_4\in \partial I^{+}(y_4)\cup \partial I^{-}(y_4)$, which contradicts the assertion). Thus, it has $\omega(x)$, $\omega(y)$ are strongly ordered.
		
	(iib) $y_*\notin I^{+}(x_*)\cup I^{-}(x_*)$. One can obtain that 
    \begin{equation}\label{x2_approx_omegay}
		\omega(y)\cap (I^{+}(x_*)\cup I^{-}(x_*))=\emptyset.
	\end{equation}
    (Otherwise, one can find $y_5\in\omega(y)$ such that $y_5\in I^{+}(x_*)\cup I^{-}(x_*)$. Note also that $y_*\notin I^{+}(x_*)\cup I^{-}(x_*)$. Then it follows from the connectedness of $\omega(y)$ that there is $y_6\in\omega(y)$ such that $y_6\in \partial I^{+}(x_*)\cup \partial I^{-}(x_*)$, contradicting the assertion.) 
    Then, for any $x'\in\omega(x)$, it has 
    \begin{equation}\label{x3_approx_omegay}
        \omega(y)\cap (I^{+}(x')\cup I^{-}(x'))=\emptyset.
    \end{equation} (Otherwise, there exist $x_5\in\omega(x)$ and $y_6\in\omega(y)$ such that $y_6\in I^{+}(x_5)\cup I^{-}(x_5)$, that is, $x_5\in I^{+}(y_6)\cup I^{-}(y_6)$. Note also $y_6\notin I^{+}(x_*)\cup I^{-}(x_*)$ by \eqref{x2_approx_omegay}, that is, $x_*\notin I^{+}(y_6)\cup I^{-}(y_6)$. 
    Then, by the connectedness of $\omega(x)$, one can obtain $x_6\in\omega(x)$ such that $x_6\in \partial I^{+}(y_6)\cup \partial I^{-}(y_6)$, a contradiction to the assertion). Furthermore, we show that $\omega(y)\cap (\overline{I^{+}(x')}\cup \overline{I^{-}(x')})=\emptyset$ for any $x'\in\omega(x)$. (If not, there exist $x_7\in\omega(x)$ and $y_7\in\omega(y)$ such that $y_7\in \overline{I^{+}(x_7)}\cup \overline{I^{-}(x_7)}$ with $x_7\neq y_7$, and then $\varphi_t(y_7)\in I^{+}(\varphi_t(x_7))\cup I^{-}(\varphi_t(x_7))$ for any $t>0$ by (H2) (or rather, Proposition \ref{quasi-closed-prop}) and the strong positivity of the system, a contradiction to \eqref{x3_approx_omegay}). So, $\omega(x)$, $\omega(y)$ are unordered.
    
    Therefore, we have proved the lemma provided that the assertion is confirmed.
		
	Finally, it remains to prove the assertion in (ii). Suppose that there exist $x_{0}\in\omega(x)$ and $y_{0}\in\omega(y)$ such that $y_{0}\in \partial I^{+}(x_{0})\cup \partial I^{-}(x_{0})$ with $x_0\neq y_0$. 
    Then by (H2) (or rather, Proposition \ref{quasi-closed-prop}) and the strong positivity of the system, we obtain that $\varphi_1(y_{0})\in I^{+}(\varphi_1(x_{0}))\cup I^{-}(\varphi_1(x_{0}))$. Since $\varphi_1(x_{0})\in \omega(x)$, there are $x_8\in O^+(x)$ close enough to $\varphi_1(x_{0})$ that $\varphi_1(y_{0})\in I^{+}(x_8)\cup I^{-}(x_8)$ by the continuity of $I^{+}$ and $I^{-}$. 
    Also since $\varphi_1(y_{0})\in \omega(y)$, choose $y_8\in O^+(y)$ close enough to $\varphi_1(y_{0})$ that $y_8\in I^{+}(x_8)\cup I^{-}(x_8)$. Since $x_8\notin \omega(y)$ (as $\omega(x)\cap\omega(y)=\emptyset$), according to Corollary \ref{absorb}, it has $\omega(y)\subset I^{+}(x_8)\cup I^{-}(x_8)$. 
    Thus, for any $y'\in \omega(y)$, it has  $y'\in I^{+}(x_8)\cup I^{-}(x_8)$, that is, $x_8\in I^{+}(y')\cup I^{-}(y')$. 
    So, it has $\omega(x)\subset I^{+}(y')\cup I^{-}(y')$ for any $y'\in \omega(y)$ by Corollary \ref{absorb} and the fact $y'\notin\omega(x)$ (as $\omega(x)\cap\omega(y)=\emptyset$). And hence, $x_0\in I^{+}(y_0)\cup I^{-}(y_0)$, that is $y_0\in I^{+}(x_0)\cup I^{-}(x_0)$, which contradicts with $y_{0}\in\partial I^{+}(x_{0})\cup \partial I^{-}(x_{0})$. Thus, we have proved the assertion and completed the proof.
\end{proof}

    \subsection{Proof of Lemma \ref{endpoint}}\label{s3.4}
	Now we are ready to prove the intersection principle for chronological boundary.

	\begin{proof}[Proof of Lemma \ref{endpoint}]
    (i): $x\in \partial I^{+}(x)\cup \partial I^{-}(x)$. By the Proposition \ref{chronological boundary}(ii), we have that $x\in \partial I^{+}(x)\cap \partial I^{-}(x)$.
    First, we prove $\mathcal{B}(\varphi) \cap \partial  I^{-}(x) = \{x\}$ by contradiction. Suppose there is $z\in \mathcal{B}(\varphi)\cap \partial  I^{-}(x)$ with $z\neq x$. By \textnormal{(H2)} (or rather, Proposition \ref{quasi-closed-prop}) and the strong positivity of the system, one can find a $\tau>0$ such that $\varphi_{\tau}(z)\ll_M\varphi_{\tau}(x)$ and $\varphi_{-\tau}(z)\notin \overline{I^{-}(\varphi_{-\tau}(x))}$. 
		
	Take a sequence $\{z_n\}_{n\geq1} \subset \mathcal{R}(\varphi)$ satisfying $z_n \rightarrow z$ ($\varphi_{\tau}(z_n)\rightarrow\varphi_{\tau}(z)$) as $n\rightarrow \infty$.
	Without loss of generality, we assume that $\varphi_{\tau}(z_{n})\ll_M\varphi_{\tau}(x)$ for all $n\geq1$. 
	Noticing $x,z_n\in \mathcal{R}(\varphi)$, one has $\varphi_{\tau}(x), \varphi_{\tau}(z_n)\in \mathcal{R}(\varphi)$. Then, by Lemma \ref{limit_set_dichotomy} for $\varphi_{\tau}(x)$ and $\varphi_{\tau}(z_n)$, we obtain that $\omega(x)$, $\omega(z_n)$ are strongly ordered for any $n\geq1$.
		
	On the other hand, there is a subsequence of $\{z_n\}_{n\geq1}$, which still labeled by $\{z_n\}_{n\geq1}$, such that $\varphi_{-\tau}(z_n)\notin \overline{I^{-}(\varphi_{-\tau}(x))}$ for all $n\geq1$.
	Then by Lemma \ref{limit_set_dichotomy} again, we obtain that $\omega(x)$, $\omega(z_n)$ are unordered for all $n\geq1$, a contradiction. Thus, we obtain that $\mathcal{B}(\varphi) \cap \partial  I^{-}(x) = \{x\}$. 
    
    We can obtain $\mathcal{B}(\varphi) \cap \partial  I^{+}(x) = \{x\}$ by analogy. Thus, we have proved that $\mathcal{B}(\varphi)\cap (\partial  I^{+}(x)\cup \partial  I^{-}(x))=\{x\}$ if $x\in \partial I^{+}(x)\cup \partial I^{-}(x)$.
     
    (ii): $x\in I^{+}(x)\cup I^{-}(x)$. By the Proposition \ref{chronological boundary}(i), we have that $x\in I^{+}(x)\cap I^{-}(x)$. We prove $\mathcal{B}(\varphi) \cap \partial  I^{-}(x) = \emptyset$ by contradiction. Suppose there is $z\in \mathcal{B}(\varphi)\cap \partial  I^{-}(x)$. Clearly, $z\neq x$ (since $x\in I^{-}(x)$ by the Proposition \ref{chronological boundary} (i) and the assumption $x\in I^{+}(x)\cup I^{-}(x)$). By \textnormal{(H2)} (or rather, Proposition \ref{quasi-closed-prop}) and the strong positivity of the system, one can find a $\tau'>0$ such that $\varphi_{\tau'}(z)\ll_M\varphi_{\tau'}(x)$ and $\varphi_{-\tau'}(z)\notin \overline{I^{-}(\varphi_{-\tau'}(x))}$. With a similar argument in Case (i), one can take a sequence $\{z_n\}_{n\geq1} \subset \mathcal{R}(\varphi)$ such that $\omega(x)$, $\omega(z_n)$ are strongly ordered for any $n\geq1$. And one can also take a subsequence $\{z_{n_i}\}_{n_i\geq1}$ of $\{z_n\}_{n\geq1}$ such that $\omega(x)$, $\omega(z_{n_i})$ are unordered for any ${n_i}\geq1$, a contradiction. Thus, we obtain that $\mathcal{B}(\varphi) \cap \partial  I^{-}(x) = \emptyset$. Similarly, $\mathcal{B}(\varphi) \cap \partial  I^{+}(x) = \emptyset$ can be obtained by analogy. Thus, we have proved that $\mathcal{B}(\varphi)\cap (\partial  I^{+}(x)\cup \partial  I^{-}(x))=\emptyset$ if $x\in I^{+}(x)\cup \partial I^{-}(x)$.
    
    Therefore, we have completed the proof.
	\end{proof}
	
	\noindent\section{Proof of Theorem \ref{Thm_2.1}}\label{Limit dich}

    In this section, we focus on the characterization of the order structure for $\mathcal{B}(\varphi)$ and prove the main theorem. The proof of Theorem \ref{Thm_2.1} heavily depends on the intersection principle (Lemma \ref{endpoint}), which has been established in section \ref{s3}. We present the detail proof as follows:
	
	\begin{proof}[Proof of Theorem \ref{Thm_2.1}]
		We prove this theorem by contradiction. Suppose that $B$ is neither unordered nor strongly ordered. First, we assert that
		\begin{equation}\label{equ-1}
			 \text{there is } x\in \mathcal{R}(\varphi)\cap B \text{ such that } B\cap (I^+(x)\cup I^-(x))\neq\emptyset \text{ and } B\cap (M\backslash \overline{I^+(x)\cup I^-(x)})\neq\emptyset.
		\end{equation}
		We only need to prove that 
		\begin{equation}\label{equ-2}
		\text{there is } a\in B \text{ such that } B\cap (I^+(a)\cup I^-(a))\neq\emptyset \text{~and~} B\cap (M\backslash \overline{I^+(a)\cup I^-(a)})\neq\emptyset.
		\end{equation}
		Indeed, by the continuity of $I^{+}$ and $I^{-}$, the assertion \eqref{equ-1} follows immediately as one selects a point $x\in \mathcal{R}(\varphi)\cap B$ sufficiently close to $a$.
        Now, we begin to prove \eqref{equ-2}. It follows from the fact that $B$ is not unordered that there are $b_1,b_2 \in B$ such that $b_2\in J^{+}(b_1)\cup J^{-}(b_1)$. 
        And the fact that $B$ is not strongly ordered entails that there exist $b_3,b_4 \in B$ such that $b_3\notin I^+(b_4)\cup I^-(b_4)$. 
        By the strong positivity of the system, for any $t>0$, it has $\varphi_{t}(b_2)\in I^+(\varphi_{t}(b_1))\cup I^-(\varphi_{t}(b_1))$. 
        Together with Proposition \ref{quasi-closed-prop}, the strong positivity of the system implies that $\varphi_{-t}(b_4)\notin\overline{I^+(\varphi_{-t}(b_3))\cup I^-(\varphi_{-t}(b_3))}$ for any $t>0$. 
        Fix $\tau>0$. If $\varphi_{-\tau}(b_3)\notin\overline{I^+(\varphi_{\tau}(b_1))\cup I^-(\varphi_{\tau}(b_1))}$, then we choose $a=\varphi_{\tau}(b_1)$. 
        Thus, $\varphi_{\tau}(b_2)\in B\cap (I^+(a)\cup I^-(a))$ and $\varphi_{-\tau}(b_3)\in B\cap (M\backslash \overline{I^+(a)\cup I^-(a)})$. 
        Therefore, we obtain \eqref{equ-2}.
		While, if $\varphi_{-\tau}(b_3)\in\overline{I^+(\varphi_{\tau}(b_1))\cup I^-(\varphi_{\tau}(b_1))}$, we can take small $s\in(0,\tau)$ and then $\varphi_{-\tau+s}(b_3)\in I^+(\varphi_{\tau+s}(b_1))\cup I^-(\varphi_{\tau+s}(b_1))$.
		Recall that $\varphi_{-t}(b_4)\notin\overline{I^+(\varphi_{-t}(b_3))\cup I^-(\varphi_{-t}(b_3))}$ for any $t>0$, it follows that $\varphi_{-\tau+s}(b_4)\notin\overline{I^+(\varphi_{-\tau+s}(b_3))\cup I^-(\varphi_{-\tau+s}(b_3))}$, which allows us to choose $a=\varphi_{-\tau+s}(b_3)$. 
        So, $\varphi_{\tau+s}(b_1)\in B\cap (I^+(a)\cup I^-(a))$ and $\varphi_{-\tau+s}(b_4)\in B\cap (M\backslash \overline{I^+(a)\cup I^-(a)})$. 
        Thus, we have \eqref{equ-2}, and hence, we obtain the assertion \eqref{equ-1}.

		Let $B_{I}=B\cap  \overline{I^+(x)\cup I^-(x)}$ and $B_{K}=B\cap (M\backslash (I^+(x)\cup I^-(x)))$. Clearly, by \eqref{equ-1}, $B_{I}\neq\emptyset$ since $B\cap (I^+(x)\cup I^-(x))\subset B_{I}$, and $B_{K}\neq\emptyset$ since $B\cap (M\backslash \overline{I^+(x)\cup I^-(x)})\subset B_{K}$. Moreover, $B_{I}$ and $B_{K}$ are closed in $B$, and $B=B_{I}\cup B_{K}$. By the relationship between $x$ and $I^{+}(x)\cup I^{-}(x)$, we consider the following two situations: (i) $x\in I^{+}(x)\cup I^{-}(x)$; (ii) $x\in \partial I^{+}(x)\cup \partial I^{-}(x)$.

        Case (i): $x\in I^{+}(x)\cup I^{-}(x)$. We show this case cannot happen. In fact, by virtue of Lemma \ref{endpoint}(ii) that $B_{I}\cap (\partial I^{+}(x)\cup \partial I^{-}(x))=\emptyset$, we have $B_{I}=B\cap  (I^+(x)\cup I^-(x))$, $B_{K}=B\cap (M\backslash \overline{I^+(x)\cup I^-(x))}$, and hence, $B_{I}\cap B_{K}=\emptyset$. Consequently, we obtain two disjoint closed subset $B_{I}$ and $B_{K}$ of $B$ such that $B=B_{I}\cup B_{K}$, which contradicts the connectedness of $B$.
        
		Case (ii): $x\in \partial I^{+}(x)\cup \partial I^{-}(x)$. By virtue of Lemma \ref{endpoint}(i) that $B_{I}\cap (\partial I^{+}(x)\cup \partial I^{-}(x))=\{x\}$, we have $B_{I}\backslash\{x\}=B\cap (I^+(x)\cup I^-(x))$, as well as $B_{K}\backslash\{x\}= B\cap (M\backslash\overline{I^+(x)\cup I^-(x)})$. Hence, by the connectedness of $B$, it has $B_{I}\cap B_{K}=\{x\}$.
		Furthermore, we will show that $B_{I}$ and $B_{K}$ are both connected. 
		We only prove that $B_I$ is connected. The proof of $B_K$ is analogous. Suppose $B_I$ is not connected. Then there are two disjoint nonempty closed subsets $B_{I_{1}}$ and $B_{I_{2}}$ such that $B_{I}=B_{I_{1}}\cup B_{I_{2}}$ and $x\in B_{I_{1}}$.
		Noticing that $B_{I_{2}}\cap B_{K}=\emptyset$,
		we have $(B_{I_{1}}\cup B_{K})\cap B_{I_{2}}=\emptyset$.
		Consequently, one can obtain two nonempty disjoint closed subsets $(B_{I_{1}}\cup B_{K})$ and $B_{I_{2}}$ of $B$ such that $B=(B_{I_{1}}\cup B_{K})\cup B_{I_{2}}$, which contradicts the connectedness of $B$. Thus, we have proved that $B_I$ is connected.
		
		Now, by virtue of \eqref{equ-1}, there exists $c\in B\cap (I^+(x)\cup I^-(x))$, i.e., $x\in I^+(c)\cup I^-(c)$ and $c\in B_{I}$. Together with the connectedness of $B_{K}$ and the fact $B_{K}\backslash \{x\}\neq \emptyset$, we can choose $x'\in B_{K}\backslash\{x\}$ close enough to $x$ such that $x'\in I^+(c)\cup I^-(c)$.
		Again, choose some $x^{*}\in \mathcal{R}(\varphi)\backslash\{x\}$ close to $x'$ such that $x^{*}\in I^+(c)\cup I^-(c)$ (since $x'\in I^+(c)\cup I^-(c)$) and $x^{*}\notin \overline{I^+(x)\cup I^-(x)}$ (since  $x'\notin \overline{I^+(x)\cup I^-(x)}$). 
		Note $x^{*}\in I^+(c)\cup I^-(c)$ is equivalent to $c\in I^+(x^{*})\cup I^-(x^{*})$ which means $B_I\cap (I^+(x^{*})\cup I^-(x^{*}))\neq\emptyset$. Also noting $x^{*}\notin \overline{I^+(x)\cup I^-(x)}$ implies $x^{*}\notin I^+(x)\cup I^-(x)$, equivalently $x\notin I^+(x^{*})\cup I^-(x^{*})$, with Lemma \ref{endpoint} that $x\notin \partial  I^{+}(x^{*})\cup \partial  I^{-}(x^{*})$, it has $x\notin \overline{I^+(x^{*})\cup I^-(x^{*})}$ which means $B_I\cap (M\backslash \overline{I^+(x^{*})\cup I^-(x^{*})})\neq\emptyset$. 
		Now, we have that both $B_I\cap (I^+(x^{*})\cup I^-(x^{*}))\neq\emptyset$ and $B_I\cap (M\backslash \overline{I^+(x^{*})\cup I^-(x^{*})})\neq\emptyset$ hold.

        Let $B_1=B_{I}\cap  (I^+(x^{*})\cup I^-(x^{*}))$, $B_2=B_{I}\cap (M\backslash\overline{I^+(x^{*})\cup I^-(x^{*})})$ and $B_3=B_{I}\cap  (\partial I^+(x^{*})\cup \partial I^-(x^{*}))$. Clearly, $B_1$ and $B_2$ are nonempty and relatively open in $B_I$. Moreover, $B_i\cap B_j=\emptyset$ for $i,j=1,2,3$ with $i\neq j$.
        If $x^{*}\in (I^+(x^{*})\cup I^-(x^{*}))$, then by Lemma \ref{endpoint}(ii), it has $B_3 = \emptyset$, which implies there are two nonempty disjoint relatively open subsets $B_1$ and $B_2$ satisfying $B_{I}=B_1\cup B_2$, a contradiction to the connectedness of $B_{I}$. 
        If $x^{*}\in (\partial I^+(x^{*})\cup \partial I^-(x^{*}))$, then by Lemma \ref{endpoint}(i), one has $\mathcal{B}(\varphi)\cap (\partial  I^{+}(x^{*})\cup \partial  I^{-}(x^{*}))=\{x^{*}\}$. Since $B_I\subset \mathcal{B}(\varphi)$ (thus, $B_3\subset\mathcal{B}(\varphi)\cap (\partial  I^{+}(x^{*})\cup \partial  I^{-}(x^{*}))$), we obtain $B_3 =\emptyset$ or $B_3=\{x^{*}\}$. Recalling $x^{*}\notin \overline{I^+(x)\cup I^-(x)}$, it has $x^{*}\notin B_I$, and hence, $x^{*}\notin B_3$. Therefor, it has $B_3=\emptyset$ which implies there are two nonempty disjoint relatively open subsets $B_1$ and $B_2$ satisfying $B_{I}=B_1\cup B_2$, a contradiction to the connectedness of $B_{I}$.

        Thus, we have proved that $B$ satisfies the order-dichotomy.
	\end{proof}

\end{document}